\documentclass[11pt]{amsart}

\usepackage[textwidth=16cm, textheight=23cm,marginratio=1:1]{geometry}
\usepackage{tikz-cd}
\usepackage {amssymb}
\usepackage{amsfonts}
\usepackage{amsmath}
\usepackage{amsthm}
\usepackage{tikz-cd}
\usepackage[enableskew]{youngtab}
\usepackage{nicefrac}
\usepackage{ytableau}
\usepackage{xcolor}
\usepackage{tabularray}
\usepackage{nicematrix}
\usepackage{abstract}
 \usepackage{booktabs}
 \usepackage{float}
\usepackage[bookmarks,colorlinks=true,%
pdfview=FitH,pdfstartview=FitH,%
urlcolor=blue,pdftitle={Hilbert-Burch Matrices},%
pdfauthor={Piotr Oszer},%
pdfsubject={},%
pdfcreator={LaTeX},%
pdfkeywords={}%,
]{hyperref}

\def\ab{\mathbf{a}}

\def\A{\mathbb A}
\def\a{\alpha}
\def\m{\mathfrak m}
\def\b{\beta}
\def\C{\mathbb C}
\def\Z{\mathbb Z}

\def\G{\mathbb G}

\def\Q{\mathbb Q}

\def\V{V}

\def\O{\mathcal O}
\def\epsilon{\varepsilon}

\DeclareMathOperator{\Lt}{Lt}

\DeclareMathOperator{\Spec}{Spec}
\DeclareMathOperator{\depth}{depth}

\DeclareMathOperator{\codim}{codim}
\DeclareMathOperator{\Hilb}{Hilb}

\DeclareMathOperator{\res}{res}

\DeclareMathOperator{\ord}{ord}

\newcommand{\kk}{\Bbbk}%

\newtheorem{lem}{Lemma}[section]
\newtheorem{defi}[lem]{Definition}
\newtheorem{defi-not}[lem]{Definition/Notation}
\newtheorem{twr}[lem]{Theorem}
\newtheorem{stwr}[lem]{Corollary}
\newtheorem{exm}[lem]{Example}
\newtheorem{quest}[lem]{Question}
\newtheorem{remark}[lem]{Remark}

\newcommand{\thefuturetheoreminner}{} % initialize

\newcommand{\thefuturexminner}{} % initialize

\newtheorem*{twri}{Theorem}

\newtheorem*{exmi}{Example}
\newtheorem*{questi}{Question}

\newtheorem*{notation}{Notation}

\hypersetup{
	colorlinks,
	linkcolor={red!80!black},
	citecolor={green!80!black},
	urlcolor={blue!80!black}
}

	\title[Hilbert-Burch matrices and explicit families on $\Hilb^d(\A^2)$]{Hilbert-Burch matrices and explicit torus-stable families of finite subschemes of $\A^2$}
\author{Piotr Oszer}	

\address{Institute of Mathematics, University of Warsaw, Banacha 2, 02-097
	Warsaw, Poland}
\email{po394403@mimuw.edu.pl}
\thanks{The author is supported by National Science Centre grant 2020/39/D/ST1/00132. The author would like to thank the Isaac Newton Institute for Mathematical Sciences, Cambridge, for support and hospitality during the programme New equivariant methods in algebraic and differential geometry, where work on this paper was undertaken (supported by EPSRC grant EP/R014604/1). The author would also like to thank Lie Fu, Joachim Jelisiejew, Jakub Koncki, Michael Thaddeus and Dimitri Wyss for helpful discussions and suggestions.
The author thanks the anonymous referee for helpful comments.
}

\begin{document}

	\maketitle
	%%TO DO 
	% Zmień przykład żeby to mówiło o rezolwencie z Hilberta-Burcha
	% tekst łączący
	\begin{abstract}
		We give an explicit description of the Białynicki-Birula cells on the Hilbert scheme of points on $\A^2$ with isolated fixed points, with the Hilbert-Burch matrices as a main tool. If the fixed point locus is positive dimensional we obtain an \'etale rational map to the cell. We prove Conjecture \cite[Conjecture 4.2]{RoserWinz-2} which we realize as a special case of our construction. We also show examples when the construction provides a rational \'etale map to the Hilbert scheme which is not contained in any Białynicki-Birula cell. Finally, we give an explicit description of the formal deformations of any ideal in the Hilbert scheme of points on the plane.
	\end{abstract}
	\section{Introduction}
		The Hilbert scheme of $d$-points on a surface is a moduli space of closed subschemes of finite length lying on the surface. It is a classical object in algebraic geometry, very well studied and admits a large number of applications, e.g. \cite{ES1,Evain,ESG,Haiman-catalan}.
		In this article, we are interested in the Hilbert scheme of points on the plane. 
		\vskip 0.2 cm
		Ellingsrud and Str\o mme in \cite{ES1} have computed Betti numbers of the Hilbert scheme by constructing a particular basis of homology coming from the Białynicki-Birula cells. To do so, they have studied $\G_m^2$-fixed points on $\Hilb^d(\A^2)$, the fixed points on the level of algebra are the monomial ideals of colength $d$.
		They described the tangent space in terms of first-order deformations of the Hilbert-Burch matrix. Such a matrix induces a free resolution of the monomial ideal.
		Subsequently in \cite{ES-2} they have exploited those techniques further and claimed to describe the Białynicki-Birula cells explicitly. Some issues with arguments in that paper were pointed out and addressed in \cite{ES-2fix}, as a consequence the description for one-dimensional tori with isolated fixed points was completed.
		Prior to that, Laurent Evain, coming from a different perspective, has fixed these issues in some cases, see \cite[Remark 16]{Evain}.
		\vskip 0.2 cm
		Since the foundational work of Ellingsrud and Str\o mme  many people have studied points of this Hilbert scheme explicitly, see \cite{Aldo-HB,Alexandru,RoserWinz-1,RoserWinz-2}. The main method revolves around deforming the Hilbert-Burch matrix for a given monomial ideal.
		This way, for a given monomial ideal $E$, Conca and Valla described  Gr\"obner cells centred at $E$ for the lexicographic term order.
		Later Roser Homs and Anna-Lena Winz, motivated by the study of the punctual Hilbert scheme, studied Gr\"obner cells for the negative degree lexicographical monomial order. In \cite{RoserWinz-2} they have described explicitly those cells for lex monomial ideals, see \cite[Theorem 5.7]{RoserWinz-1}. They have conjectured that the description extends to Gr\"obner cells for any monomial ideal $E$, see \cite[Conjecture 4,2]{RoserWinz-2}.
		\vskip 0.cm
		It remained open to construct all the Białynicki-Birula/Gr\"obner cells explicitly in terms of Hilbert-Burch matrices. There is also a question when such constructions provide us with explicit open subschemes of the Hilbert scheme.
		In this paper, we answer these questions.\\ Choose a monomial ideal $E=(x^{t},x^{t-1}y^{m_1},\ldots, xy^{m_{t-1}}, y^{m_t})$ and take ${d_i=m_{i}-m_{i-1}}$, with the convention that $m_0=0$. We consider the \emph{spread-out} matrix:
			$$ \mathcal M :=
		\begin{pmatrix}
			y^{d_1}+a_{1,1} & a_{1,2} & a_{1,3} & \ldots & a_{1,t} \\
			-x+a_{2,1} & y^{d_2}+a_{2,2} & a_{2,3} & \ldots & a_{2,t} \\
			a_{3,1} & -x+a_{3,2} & y^{d_3}+a_{3,3} & \ldots & a_{3,t} \\
			
			\vdots & \vdots & \vdots  & \ddots & \vdots \\
			a_{t,1} & \ldots & \ldots  & -x+a_{t,t-1} & y^{d_t}+a_{t,t} \\
			a_{t+1,1} & \ldots & \ldots  & a_{t+1,t-1} & -x+a_{t+1,t} \\
		\end{pmatrix},
		$$	
		where $a_{i,j} \in \kk[y]$ and $$\begin{cases}
			\deg_y(a_{i,j})< 	d_j, &  \text{for } i>j\\
			\deg_y(a_{i,j})< 	d_i, & \text{for } i\leq j,
		\end{cases}$$ where $a_{i,j}=\sum_{k} a_{i,j,k}y^k$, we introduce a notation $\ab= (a_{i,j,k})$. The Hilbert-Burch matrix associated to $E$ is the matrix with coefficients in $\kk[x,y]$ obtained from $\mathcal M$ by putting all parameters $a_{i,j,k}$ to zero.
		The spread-out matrix defines a family $\pi$ given by the ideal of maximal minors $I_t(\mathcal M)$
			$$\begin{tikzcd}
			\V(I_t(\mathcal M)) \arrow[r, hook, "cl"] \arrow[rd, "\pi"] & \A^2 \times \Spec(k[\ab])  \arrow[d] \\
			& \Spec(\kk[\ab]).                        
			  \end{tikzcd}$$
		 We show that this family controls all infinitesimal deformations of the ideal $E$.
		\begin{twri}[{Theorem~\ref{infinitesimal-phi}}]
		Let $E$ be a monomial ideal of colength $d$. The matrix $\mathcal M$ via the family $\pi$ induces an isomorphism of spectra of completed local rings
		$$\hat d\phi_E\colon \Spec(\kk[[\ab]]) \to \Spec(\widehat{\O_{\Hilb^d,[E]}}).  $$
		\end{twri}
		The matrix $\mathcal M$ is defined over $B:=\Spec \kk[\ab]$, but the formal-local situation, described above, in general does not extend to the global one. The family $\pi$ in general is not finite and flat.
%	\begin{exmi}
	%	{Example \ref{exm-nonflatness}}
	%	Let us take a monomial ideal $E=(y^2,x^2)$. The Hilbert-Burch matrix %associated to $E$ is 
	%	$$ M_E=
	%	\begin{pmatrix}
	%		y^{2} & 0 \\
	%		-x 	  & 1 \\
	%		0 & -x 
	%	\end{pmatrix}.
	%	$$
	%	One of the deformations of this matrix that is coming from the %construction is
	%	$$
	%	\mathcal M=	\begin{pmatrix}
	%		y^{2} & y \\
	%		-x+y 	  & 1 \\
	%		0 & -x 
	%	\end{pmatrix}. 
	%	$$
	%	The $2\times 2$ minors of this matrix have a common factor $x$. Thus by %the Hilbert Burch theorem \cite[Theorem 20.15a]{eisenbud}, the complex
	%	$$ 0 \to \kk[x,y]^{\oplus 2} \xrightarrow{\mathcal M}  \kk[x,y]^{\oplus %3} \to I_t(M) \to 0$$
	%	where $I(M)$ is an ideal generated by $2\times 2$ minors, is not exact.
		
	%	\end{exmi}

	\begin{exmi}[{Example~\ref{exm-non-finitness}}]
		Let us take a monomial ideal $E=(y^3,x^2)$. The Hilbert-Burch matrix associated to $E$ is 
		$$ M_E=
		\begin{pmatrix}
			y^{3} & 0 \\
			-x 	  & 1 \\
			0 & -x 
		\end{pmatrix}. 
		$$
		One of the restrictions of the spread-out matrix of $E$ is
		$$ \mathcal M=
		\begin{pmatrix}
			y^{3} & y^2 \\
			-x+y^2	  & 1 \\
			0 & -x 
		\end{pmatrix}.
		$$
		One can check that the ideals generated by $2\times 2$ minors of $M_E$ and $\mathcal M$ have different colengths, thus the family cannot induce a morphism to the Hilbert scheme of points.
	\end{exmi} 
		We introduce a $\G_m^2$-action on $B$ making the family $\pi$ equivariant.
		Using equivariant techniques, we describe an explicit $\G_m^2$-stable subscheme $S$ of $B$ such that the family $\pi$ restricted to $S$ induces a morphism of the Hilbert scheme, see Corollary \ref{lemma-finitness}.
		\vskip 0.2 cm
		Let us choose a cocharacter $\psi\colon \G_m\to \G_m^2$. We consider the $\G_m$-action on the Hilbert scheme associated to $\psi$
		\begin{align*}
			\G_m \times \Hilb^d(\A^2) &\to \Hilb^d(\A^2),\\
			(t,[I]) &\mapsto (\psi(t)\cdot [I]).
		\end{align*}
		Consider the fixed point locus $\Hilb^{\G_m}(\A^2) $ together with its decomposition into connected components  $\Hilb^{\G_m}(\A^2) = \bigsqcup F_i$. The Białynicki Birula cell associated to $\psi$ and $F_i$ is a locally closed subscheme of $\Hilb^d(\A^2)$ given by
		 \[ \{[I]\in \Hilb^d(\A^2)\text{ , }\lim_{t\to 0} t\cdot [I] \in F_i\}. \]
		The cell centred at $[E]$ is the one that contains the chosen monomial ideal $E$. In that case, this point is the only $\G_m^2$-fixed point of the cell. \vskip 0.2 cm
		The following two results give us an explicit rational $\G_m$-equivariant \'etale map to the Białynicki-Birula cells of $\Hilb^d(\A^2)$. We consider two cases according to the degree $y$ with respect to the cocharacter $\psi$.
		First, we describe the Białynicki-Birula cell for the $y$-axis being contracted to $0$ by the $\G_m$-action associated to the chosen cocharacter.
	\begin{twri}[{Theorem~\ref{theorem-abb}}]
		Choose a cocharacter $\psi\colon \G_m \to \G_m^2$, such that $y$ has a non-negative degree with respect to the induced grading.
		Let us restrict $B$ to $B^{+,\psi}$ the Białynicki-Birula cell associated to $\psi$. The morphism $\pi$ induces a rational $\G_m^2$-equivariant map $\phi^+ \colon B^{+,\psi}~\to~\Hilb^{d}(\A^2)$. The image of $\phi^+$ lies in the Białynicki-Birula cell associated to $\psi$ and centred at $[E]$. The map $\phi^+$ is \'etale in a $\G_m^2$-stable neighbourhood of the origin.
	\end{twri}
		In case of the degree of $y$ being negative, we need to modify our construction, by intersecting the family with a thickened $x$-axis. We do that to force the support of the subscheme to be set-theoretically contained in the $x$-axis
			 $$\begin{tikzcd}
				\V(I_t(\mathcal M)+(y^{tm_t})) \arrow[r, "cl", hook] \arrow[rrd, "\pi_-"'] & \V(I_t(\mathcal M)) \arrow[r, "cl", hook] \arrow[rd, "\pi"] & \A^2 \times B \arrow[d] \\
				&                                                             & B.                      
			\end{tikzcd}$$ 
	For the detailed discussion of morphism $\pi_-$ we refer to the discussion of Figure \ref{thicken-x-axis} in Section \ref{chapter7}.
	\begin{twri}[{Theorem~\ref{theorem-abb-}}]
		Let us choose a cocharacter $\psi\colon \G_m \to \G_m^2$, such that $y$ has negative degree with respect to the induced grading.
		Let us restrict $B$ to $B^{+,\psi}$, the Białynicki-Birula cell associated to $\psi$. Then the map $\pi_{-}$ induces a rational $\G_m^2$-equivariant map $\phi^+ \colon B^{+,\psi} \to \Hilb^{d}(\A^2)$. The image of $\phi^+$ lies in the Białynicki-Birula cell associated to $\psi$ and centred at $[E]$. The map $\phi^+$ is \'etale in  a $\G_m^2$- stable neighbourhood of the origin.
	\end{twri}
	We apply the theorems above to generalize Theorem \ref{infinitesimal-phi} from only $\G_m^2$-fixed points, to any point of the Hilbert scheme, that is to any ideals.
		\begin{twri}[{Theorem~\ref{infinitesimal-phi2}}]
		Let us take an ideal $[I]\in \Hilb^d(\A^2)$. Every monomial ideal $E$ belonging to the closure of the $\G_m^2$-orbit of $[I]$ induces an isomorphism of spectra of complete local rings.
		$$ \widehat{d\phi_{I,E}} \colon \Spec(\kk[\ab]) = \Spec(\widehat{\O_{B_E,b}})\to \Spec(\widehat{\O_{\Hilb^d,I}}),$$
		where $b\in B_E$ is a point such that the ideal of maximal minors of the matrix $\mathcal M_E$ over $B_E$ is equal to $I$.
	\end{twri}
	\vskip 0.2 cm
	Our construction puts all Białynicki-Birula cells in a coherent setting. Using the dictionary between Gr\"obner picuture and $\G_m$-equivariant picture Table \ref{table}, Table \ref{table2} in Section \ref{preliminaries}, we see that the construction generalizes the one in \cite{Aldo-HB} and  solves \cite[Conjecture 4.2]{RoserWinz-2}.
	If we consider a dominating Białynicki-Birula cell centred ay $[E]$, then Theorem \ref{theorem-abb} provides us with an explicit description of an open subset of the Hilbert scheme. However, the condition that $[E]$ is the centre of a dominating Białynicki-Birula cell is not necessary, as shown in the following example. 
		\begin{exmi}[{Example~\ref{example-200}}]
			Let us consider the monomial ideal $E=(y^2,x^3)$. Let us plot the weights of a few of the parameters.
		$$	\includegraphics[scale=0.5]{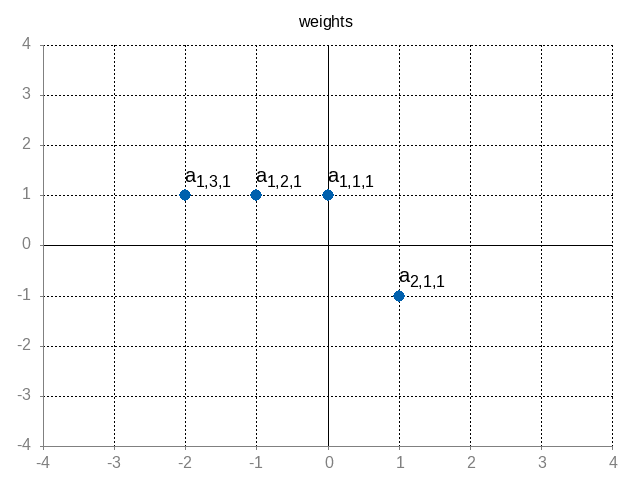} $$
			We see that there is no half-plane through the origin containing all the weights. This means that no open neighbourhood of the origin is fully contained in a Białynicki-Birula cell for some cocharacter $\psi$. Thus the potential existence of the rational map to the Hilbert scheme cannot be predicted by Theorem \ref{theorem-abb}. 		
		But one can check that the construction from Definition \ref{matrix-general} provides us with an open subscheme of the Hilbert scheme centred at $[E]$.
				\end{exmi}
		Motivated by the example above we pose the following question.
		\begin{questi}[{Question~\ref{quest}}]
				What are the monomial ideals $E$ such that the construction from Definition \ref{matrix-general} induces a rational map to Hilbert scheme, but there is no dominating Białynicki-Birula cell containing $[E]$. 
			Can one distinguish specific infinite classes of such ideals?
		\end{questi}
		The answer for that question would provide a collection of explicit open subschemes of the Hilbert scheme, coming from the original strategy by Ellingsrud and Str\o mme. Such open subschemes should be, in principle, distinct from those constructed in \cite[Proposition~2.1]{Haiman-catalan}.
		\vskip 0.2 cm
		The structure of the paper is as follows. After preliminaries, we begin in Section \ref{chapter2}, by introducing Hilbert-Burch matrices and the spread-out matrix $\mathcal M$.  This matrix defines a family of ideals that is the central object of our study. Then in Section \ref{chapter3}, we endow the construction with a $\G_m^2$-action. Next in  Section \ref{chapter4} we consider the local situation and prove Theorem \ref{infinitesimal-phi}. From that point, we consider the global question. In Section \ref{chapter5} we show that in general the family does not extend well. We develop a simple technique that allows us to control the well-definiteness of extended family. As a key technical idea, we consider a particular resultant of the outermost minors of the spread-out matrix. This resultant is a polynomial of one variable giving us an explicit criterion of well-definiteness of the morphism to the Hilbert scheme in terms of its coefficients, see Lemma \ref{resultant-flat}. Finally in Section \ref{chapter7} we apply our methods to describe the Białynicki-Birula cells. Afterwards in Section \ref{chapter tangent2} we apply the obtained results to extend Theorem \ref{infinitesimal-phi} from the case of monomial ideals to the case of arbitrary ones, see Theorem \ref{infinitesimal-phi2}.
		In the last part Section \ref{chapter8} we open a further discussion on the topic.
		\newpage
		\section{Preliminaries} \label{preliminaries}
		In this section, we state a few prerequisites. \vskip 0.2 cm
		Let us fix a field $\kk$.
		The fundamental role in the paper is played by the Hilbert-Burch theorem.
		% The whole paper can be seen as an involved application of this theorem in the case of the base ring $R$ being two-dimensional $\kk$-algebra.
		\begin{notation}
			The $i$-th maximal minor of the $(t+1)\times t$ matrix $\mathcal M$ is the determinant of $M$ with the {$i$-th} row excluded.
		\end{notation}
	
		\begin{twr}[Hilbert-Burch Theorem]	\label{Hilbert-Burch}
			Let $R$ be a Noetherian local ring.
			\begin{itemize}
				\item \cite[Theorem 20.15b]{eisenbud} Suppose we are given a $(t+1)\times (t)$ matrix $\mathcal M$ with entries in $R$ such that the depth of the ideal $I$ generated by the $t$-minors is equal to $2$. The sequence 
				$$ 0 \to R^{\oplus t} \xrightarrow{\mathcal M} R^{ \oplus  t+1} \xrightarrow{f} R \to R/I  \to 0,$$
				where the $i$-th entry of $f$ is $(-1)^i$ times the $i$-th $t$-minor of $\mathcal M$, is a free resolution.
				\item  \cite[Theorem 8.3, Schaps]{hardef}
				Additionally, assume that $R$ is a $\kk$-algebra and let $A$ be an Artinian $\kk$-algebra. If $\mathcal M'$ is any lifting of $\mathcal M$ to $R'=R\otimes_\kk A$ and $f'$ is obtained as above, then the sequence 
				$$ 0 \to (R')^{\oplus t} \xrightarrow{\mathcal M'} (R')^{\oplus  t+1} \xrightarrow{f'} R' \to R'/I'  \to 0,$$
				where $I'$ is the ideal generated by the $t$-minors of $\mathcal M'$, is again a free resolution, and $R'/I'$ is flat over $A$.
				\item \cite[Theorem 20.15a]{eisenbud}
				Conversely to the first bullet, if a complex 
					$$ 0 \to R^{\oplus t} \xrightarrow{\mathcal M} R^{ \oplus  t+1} \xrightarrow{f} R \to R/I  \to 0,$$
					is exact, then there is a nonzero divisor $z\in R$ such that $I=zI_t(\mathcal M)$, where $I_t(\mathcal M)$ is the ideal generated by the maximal minors. Then the $i$-th entry of the matrix $f$  is {$(-1)^{i}z$~ times ~the ~$i$-th~ minor~ of~ $\mathcal M$}. The ideal $I_t(\mathcal M)$ has the depth equal to 2.
			\end{itemize}
		\end{twr}
		%\begin{twr}\label{ABB-theorem}
		%	ABB decomposition 
		%\end{twr}
		In Section \ref{chapter6} we will need the following classical results on resultant.
	%	We state standard properties of resultant that are heavily used from Section \ref{chapter6}.
		{\begin{lem} [{\cite[ Chapter 3, §6 Proposition 1, Proposition 3]{Cox-book}}]
			Let $R$ be a $\kk$-algebra. Let $f,g \in R[x]$. Then
			\begin{enumerate}
				\item  The resultant $\res_x(f,g)$ of $f,g$ with respect to $x$ belongs to the ideal $(f,g) \cap R$.
				\item $\res_x(f,g)=0$ if and only if $f$ and $g$ admits a common factor which has a positive degree with respect to $x$.
			\end{enumerate}
		\end{lem}}
		In general the resultant is not stable under base change, see \cite[Chapter 3, §6]{Cox-book}, but in our case, the left argument will be a monic polynomial with respect to $x$. \vskip 0.2 cm
		\begin{lem} [{\cite[Chapter 3, §6 Proposition 3]{Cox-book}}] \label{res-stable}
			Let $f,g \in R[x]$ and $f$ be a monic polynomial of the variable~$x$. 
			The resultant is stable under base change of $R$,
			for every $\kk$-algebra homomorphism $R \to S$, image of resultant~$\res_x(f,g)$ in $S[x]$ is equal to the resultant of images of $f,g$ in $S[x]$.
			%	Let $f,g \in k[x_1,\ldots x_n]$ and $f$ be monic with respect to $x_1$,
			%	then $Res(f,g,x) \in k[x_2,\ldots x_n]$ is stable under specialization of $k[x_2,\ldots x_n]$. \\ By which we mean, that the image of $Res(f,g,x)$ in $k[x_2,\ldots x_n]/I$ for any given ideal $I \subset k[x_2,\ldots x_n]$ is exactly resultant of images of $f,g$ in ring $k[x_1,\ldots x_n]/(I \cdot k[x_1,\ldots x_n]) $.
		\end{lem}
	 	\vskip 0.2 cm
	 	Since a considerable part of the motivation of this work comes from commutative algebra and the study of Hilbert scheme from the Gr\"obner basis angle, for example \cite{Aldo-HB}, we discuss the relation between monomial orders and $\Z$-gradings on the polynomial ring $\kk[x,y]$. For reference on monomial orders see \cite[Section 1.2]{Greuel-book}.
			\begin{table}[H] 
			\caption{Dictionary between graded structures and monomial orders} \label{table} 
	 		\begin{tabular}[!htp]{cc}
	 			\addlinespace
	 			$\Z$-grading & Monomial ordering \\
	 			\toprule
	 			\addlinespace
	 			 a $\Z$-grading structure $\deg$ on $k[x,y]$ &  a monomial order $\succeq$ on $\kk[x,y]$\\
	 			\addlinespace
	 				 			\midrule
	 			\addlinespace
	 			$D=\{(\alpha-\alpha',\beta-\beta')\text{ , } \deg(x^\alpha y^\beta)>\deg(x^{\alpha'} y^{\beta'})  \}$ & $D_{\succeq}=\{(\alpha-\alpha', \beta-\beta')\text{ , } x^\alpha y^\beta \succeq	 x^{\alpha'} y^{\beta'}  \}$ \\
	 			\addlinespace
	 			\midrule
	 			\addlinespace
				$\psi(\alpha,\beta)=\alpha\deg(x)+\beta\deg(y)$ & a functional $\psi\colon \Z^2 \to \Z$ \\	
				& such that $\psi(D_{\succeq}) \subset \Z_{\geq 0}$
				\\
				\addlinespace
	 		\end{tabular}
		\end{table}
		\vskip 0.2 cm
		In case of the Hilbert scheme of $d$-points we consider only finitely many monomials. Thus for a chosen monomial order $\succeq$, we can define $\psi$ that maps the set $D_{\succeq}$ to strictly positive numbers, for reference see \cite[Lemma 1.2.11]{Greuel-book}. 
		A graded structure on $\kk[x,y]$ translates to a $\G_m$-action on Hilbert scheme, and thus we obtain the following dictionary, Table \ref{table2}.\\
		Let $I\subset \kk[x,y]$ be an ideal of colength $d$, and let $J$ be a fixed monomial ideal of colength $d$.
		{	\begin{table}[h] 
			\caption{Dictionary between the limits of $\G_m$-action and the leading term ideals} \label{table2} 
			\begin{tabular}[!htp]{cc}
				$\G_m$-action & Monomial ordering \\
				\toprule
				\addlinespace
			 	$\lim\limits_{t\to 0}t \cdot [I] \in \Hilb^{d}(\A^2)$ 
			 	&
			 	$\Lt(I)=(\Lt(f)\text{, }f\in I)$\\
			 	\addlinespace
			 	\midrule
			 	\addlinespace
			 	The Białynicki-Birula cell &  The Gr\"obner cell  \\
			 	centred at $[J]$ &  $\{I\subset \kk[x,y]\text{, }\Lt(I)=J\}$.\\
				\addlinespace

			\end{tabular}
		\end{table}}
		\\
		{Where the limit in $0$ is given by taking the $\G_m$-orbit of $[I]$, extending it to $\A^1$, and then taking the fibre over $0\in \A^1$ of the induced family of ideals.}
		\begin{exm} \cite[Example 1.2.8]{Greuel-book} \label{exm-glex}
			Let us consider the negative degree lexicographical ordering on $\kk[x,y]$:
			\begin{equation} \label{glex}
				x^\a y^\beta \succeq x^{\a'}y^{\beta'} \iff \a+\b < \a'+\b' \text{ or } \a+\b = \a'+\b' \text{ and } \a>\a'.
			\end{equation} 
			We follow the dictionary from Table \ref{table}, and look for an equivalent graded structure. From the second part of Condition \eqref{glex} we get
			$$ \deg(x)<\deg(y)<\deg(1)=0.$$
			The first part of Condition $\eqref{glex}$ says that $\deg(x)$ and $\deg(y)$ should be almost equal.
			From that we see that, if we restrict ourselves to ideals of colength $d$, the graded structure
			$$ \deg(y)=-N, $$
			$$ \deg(x)= -N(1-\varepsilon),$$
			for small enough $\epsilon \in \Q_{>0}$, and $N \in \Z_{>0}$ divisible enough, such that $\epsilon N \in \Z$, induces the same cell structure according to the dictionary from Table \ref{table2} as the chosen order.
		\end{exm}
		\vskip 0.2 cm
		We end this section with an important classical lemma concerning maps of the Białynicki-Birula cells. For reference for the Białynicki-Birula cells, see \cite[Theorem 1.5]{ABB_JJLS}.
			\begin{lem} \label{ABB-simples}
			Let us consider two affine spaces $X,Y$ of the same dimension with a linear action of $\G_m$, such that both of them are the Białynicki-Birula cells of themselves, with the origin being the only $\G_m$-fixed point. Any equivariant morphism $\phi\colon X \to Y$, which is an isomorphism on tangent spaces to the origin, is an isomorphism of schemes.
		\end{lem}
		\begin{proof}
			Since $X$ has a contracting $\G_m$-action we get a natural equivariant isomorphism 
			\begin{align*}
				X &\cong T_0 X, \\
				x &\mapsto \dfrac{\mathrm{d}gx}{\mathrm{d}g}(0).
			\end{align*}
			To be precise, since $X$ is the Białynicki-Birula cell for $\G_m$, the functor $$X^{+}(S) = \{\A^1\times S \to X \text{ , } \G_m \text{-equivariant}\}$$
			is represented by $X$.
			There is a natural transformation of the functors from $X^{+}$ to
			$$X_{\epsilon}^+(S) =	\{\Spec \kk[\epsilon]/(\epsilon^2) \times S \to X \text{ , } \G_m \text{-equivariant}\}$$
			given by the closed embedding of the thicken point $\Spec \kk[\epsilon]/(\epsilon^2)$ into the line $\A^1$ at the origin 
			$$	\begin{tikzcd}
				\A^1\times S \arrow[r]                                                                    & X & {} \\
				{\Spec \kk[\epsilon]/(\epsilon^2) \times S. } \arrow[u, hook] \arrow[ru] &   &   
			\end{tikzcd} $$
			Since the fixed point locus of $X$ is just the origin, the functor $X^+_{\epsilon}$ is represented by $T_0X$. The morphism ${X \to T_0 X}$ coming from the natural transformation $X^+\to X_\varepsilon^+$ is an isomorphism because given a $S$-point of $X_\epsilon^+$ there is only one way to complete the following diagram equivariantely
			$$	\begin{tikzcd}
				\A^1\times S   \arrow[r,dotted]                                                                 & X & {} \\
				{\Spec \kk[\epsilon]/(\epsilon^2) \times S. } \arrow[u, hook] \arrow[ru] &   &   
			\end{tikzcd} $$
			The same holds for $Y$, thus we obtain the following diagram
			$$	\begin{tikzcd}
				T_0X \arrow[r, "d\phi"]                & T_0Y                 \\
				X \arrow[r, "\phi"] \arrow[u, "\cong"] & Y \arrow[u, "\cong"].
			\end{tikzcd}$$
			By the naturality, the diagram is commutative and since $d\phi$ is an isomorphism, the morphism $\phi$ is also an isomorphism.
			%	Isomorphism on tangent spaces implies that on some neighbourhood of the origin map is \'etale, since $\phi$ equivariant this neighbourhood is $\G_m^2$-stable. Every point of cell is attracted to the origin thus the only equivariant open neighbourhood of the origin is the whole cell itself. That way we obtain that $\phi$ is \'etale. Now we prove universal injectivness, by  \cite[\href{https://stacks.math.columbia.edu/tag/01S4}{Tag 01S4}]{stacks-project}, we only need to check injectivness on $K$-points, for any field $K$. Let us take a pair of $K$-points $b_1,b_2 \in X(K)$ such that $\phi(b_1)=\phi(b_2)$.
			%		Their orbits $\G_m \cdot b_1$ and $\G_m \cdot b_2$, have to be equal, otherwise it would contradict \'etalness at the origin.  
			%		We restrict $\phi$ to the orbit $\G_m \cdot b_1$ and obtain \'etale surjective equivariant map from $\Spec(K[z]) =\A^1$ to $\A^1$. That is an endomorphism of $k[z]$
			%		\begin{align*}
				%			K[z] &\to K[z] \\
				%			z &\mapsto W(z)
				%		\end{align*}
			%		Since the map is \'etale the order of vanishing of $z$ has to be $1$, but it is equviariant, so it has to be linear, thus it is an identity. 
			%			As a conseqconsequence  $\phi$ is an \'etale, bijective morphism, thus an isomorphism.
		\end{proof}
	
	\section{Hilbert-Burch matrices} \label{chapter2}
		We are going to study the Hilbert scheme of $d$ points on the affine plane, for fixed positive integer $d$.
		Let us fix $E$ to be a monomial ideal in $\kk[x,y]$, such that $\dim_{\kk}(\kk[x,y]/E)=d$. Such an ideal can be presented as $(x^{t},x^{t-1} y^{m_1},\ldots, xy^{m_{t-1}}, y^{m_t})$ for $m_i$ being a non-decreasing sequence of natural numbers. The set of generators above is minimal if and only if the sequence $m_i$ is strictly increasing. 
		Let us define $d_i=m_{i}-m_{i-1}$, with the convention that $m_0=0$.
		We follow {\cite[page 5]{Aldo-HB}} and define the Hilbert-Burch matrix for $E$
		$$
		M_E=	\begin{pmatrix}
			y^{d_1} & 0 & 0 & \ldots & \ldots & 0 \\
			-x & y^{d_2} & 0 & \ldots & \ldots & 0 \\
			0 & -x & y^{d_3} &  \ldots &\ldots & 0 \\
			\vdots & \vdots & \vdots  & \ddots & \vdots & \vdots \\
				\vdots & \vdots & \vdots  & \vdots & \ddots & \vdots \\
			0 & \ldots & \ldots  & \ldots  & -x & y^{d_t} \\
			0 & \ldots & \ldots & \ldots & 0 & -x \\
		\end{pmatrix}.
		$$
		This is a matrix of size $(t+1)\times t$. Let us observe that the set of $t$-minors of $M_E$ is up to the sign $(x^{t},x^{t-1} y^{m_1},\ldots, xy^{m_{t-1}}, y^{m_t})$.
		Let us take $R=\kk[x,y]$. 
		The Hilbert-Burch matrix represents a free resolution of the ideal as a $\kk[x,y]$-module:
		\begin{equation} \label{res-HB}
			0 \to R^{\oplus t} \xrightarrow{M_E} R^{\oplus t+1} \to E \to 0\end{equation}
		where the homomorphism $(R)^{\oplus t+1} \to E $ sends $(-1)^{i}$ times the $i$-th generator to $i$-th maximal minor of $M_E$, as in the Hilbert-Burch Theorem \ref{Hilbert-Burch}.

		\vskip 0.2 cm
		Now we are going to spread out the matrix $M_E$, by which we mean make $M_E$ a central fibre of a certain family of matrices as defined in Definition \ref{matrix-general}, in order to get a family of the Hilbert-Burch matrices. The matrix $M_E$ has two diagonals, we consider a family of matrices with entries corner-wise dominated by those diagonals. That means if we choose an element on the diagonal $(i,i)$, every entry below in the same column and to the right in the same row has a lower degree with respect to $y$. Additionally there are no entries dependent on $x$, other than $-x$ on the diagonal $(i+1,i)$ as in the matrix $M_E$. \vskip 0.2 cm
		For example, all entries beside $y^{d_2}$ in the highlighted area in the matrix below have the degree with respect to $y$ strictly lower than $d_2$
			$$	\begin{pNiceMatrix}
			
			y^{d_1} & \ast  & \cdots   & \ast  \\
			
			\ast & 	\Block[fill=red!60]{1-3}{} y^{d_2} \Block[fill=red!60]{4-1}{}  & \cdots  &  \ast  \\
			\vdots &  \vdots  & \ddots   & \vdots  \\
			\ast & \ast  & \cdots  & y^{d^t}  \\
			\ast & \ast  & \cdots     & \ast \\
			\end{pNiceMatrix}.$$
		Besides the $-x$ on the diagonal $(i+1,i)$, coming from $M_E$, there is no other entry dependent on $x$.
		\begin{defi} \label{matrix-general}

			We define the base scheme $B_E=\Spec(k[\ab])$ to be an affine space over which we have a matrix
			$$ \mathcal M_E :=
			\begin{pmatrix}
				y^{d_1}+a_{1,1} & a_{1,2} & a_{1,3} & \ldots & a_{1,t} \\
				-x+a_{2,1} & y^{d_2}+a_{2,2} & a_{2,3} & \ldots & a_{2,t} \\
				a_{3,1} & -x+a_{3,2} & y^{d_3}+a_{3,3} & \ldots & a_{3,t} \\
		
				\vdots & \vdots & \vdots  & \ddots & \vdots \\
				a_{t,1} & \ldots & \ldots  & -x+a_{t,t-1} & y^{d_t}+a_{t,t} \\
				a_{t+1,1} & \ldots & \ldots  & a_{t+1,t-1} & -x+a_{t+1,t} \\
			\end{pmatrix},
			$$	
			where $a_{i,j} \in \kk[y,\ab]$ and $$\begin{cases}
				\deg_y(a_{i,j})< 	d_j, &  \text{for } i>j\\
				\deg_y(a_{i,j})< 	d_i, & \text{for } i\leq j,
			\end{cases}$$ where $a_{i,j}=\sum_{k} a_{i,j,k}y^k$, we introduce a notation $\ab= (a_{i,j,k})$.
		\vskip 0.2 cm
		The matrix $\mathcal M_E$ is a family of matrices over $B_E$, it is the matrix $M_E$ spread out from the origin over the base ~$B_E$. From now on we will omit the subscript $E$ and denote the base $B$ and the spread-out matrix $\mathcal M$.
			We also define a map $\pi$:
		$$	\begin{tikzcd}
			\V(I_t(\mathcal M)) \arrow[r, hook, "cl"] \arrow[rd, "\pi"] & \A^2 \times B  \arrow[d] \\
			& B,                       
		\end{tikzcd}$$
		where $I_t(\mathcal M)$ is the ideal generated by the maximal minors of the matrix $\mathcal M$.
			\end{defi}
		\begin{exm} \label{first-exm}
		Let $E$ be the ideal generated by $(y^3,xy,x^2y,x^3)$, then the Hilbert-Burch matrix for $E$ is:
			$$ M_E= \begin{pmatrix}
			y& 0 &0     \\
			-x & 1     &  0           \\
			0   & -x    &  y^2 \\
			0  &  0     & -x  \\	
		\end{pmatrix}
		$$
		and our spread-out matrix is
		$$ \mathcal M= \begin{pmatrix}
		 y+a_{1,1,0} & a_{1,2,0} &a_{1,3,0}     \\
		 -x+a_{2,1,0} & 1     &  0           \\
		 a_{3,1,0}   & -x    &  y^2+a_{3,3,1}y+a_{3,3,0} \\
		 a_{4,1,0}  &  0     &  a_{4,3,1}y+a_{4,3,0}-x  \\	
					\end{pmatrix}
	$$		
		note that $a_{*,1}= a_{*,1,0}$ and the same goes for the first column. The base scheme is the spectrum of $${\kk[\ab]=\kk[a_{1,1,0},\ldots a_{1,3,0},a_{2,1,0},a_{3,1,0},a_{4,1,0},a_{3,3,1},a_{3,3,0},a_{4,3,1},a_{4,3,0}]}.$$
		The length of the quotient algebra $R/E$ is $5$. Thus $E$ naturally lies in $Hilb^5(\A^2)$  which is of dimension $10$, which is the dimension of the base $B=\Spec \kk[\mathbf a]$.
			\end{exm}
		 Based on the example above, one could expect that $B$ actually parametrizes certain resolutions coming from the Hilbert-Burch Theorem and thus controls the deformations of $\Spec(R/E)$ as a finite subscheme of $\A^2$. This speculation is investigated in Section \ref{chapter4} up to Section \ref{chapter8}.
		Infinitesimally the speculation holds, see Theorem \ref{infinitesimal-phi}. Usually the global case does not, see Examples  \ref{exm-nonflatness} and \ref{exm-non-finitness}. Before investigating those issues further, in the next section we endow the family with a $\G_m^2$-action.
	
	\section{Torus action}\label{chapter3}
		In the previous section, for a given monomial ideal $E$ we have constructed a family of matrices $\mathcal M$ over the base scheme $B$ and considered a map $\pi$ from the incidence scheme $\V(I_t(\mathcal M))$ to $B$, mapping the vanishing locus of $t$-minors of $M_b$ to the point $b\in B$, see Definition \ref{matrix-general}.  
		In this section, we endow the scheme $B$ with a $\G_m^2$-action, making our construction equivariant. The equivariance of the family significantly simplifies considerations in later sections. We start by recalling the equivariant resolution of $E$.
		\vskip 0.2 cm
		\begin{lem} [{\cite[Lemma 3.1]{ES1}}]\label{simple-resolution}
			There is a $\G_m^2$-equivariant structure on Resolution \eqref{res-HB} with a bi-graded structure.
			\begin{equation} \label{gm-res-simple} 0 \to \bigoplus^{t}_{j=1} R[-t+j-1,-m_j] \xrightarrow{M_E} \bigoplus^{t+1}_{i=1}R[-t+i-1,-m_{i-1}] \to E \to 0,\end{equation}
			where $R=\kk[x,y]$ is endowed with a bi-degree $\deg(x^ay^b)=(a,b)$.
		\end{lem}
		\begin{proof}
			Let us check that $M_E$ is homogeneous. Then non-zero entries of $M_E$ are either $y^{d_i}$ on the $(i,i)$-th place or $-x$ on the $(i+1,i)$-th. 
			Let us recall that from the definition $m_i=\sum_{j\leq i} d_j$.
			In the first case $$(-t+i-1,-m_i)+(0,d_i)=(-t+i-1,-m_i+d_i)=(-t+i-1,-\sum_{j<i} d_j) = (-t+i-1,-m_{i-1})$$
			which is the degree of the $i$-th component of the target module. \\
			In the second case
			$$ (-t+i-1,-m_i)+(1,0)= (-t+(i+1)-1,-m_{(i+1)-1})$$
			which is the degree of $(i+1)$-th component of the target module.
			Therefore the homomorphism of modules $M_E$ is homogeneous. To complete the proof we notice that maximal minors of $M_E$ are $(x^{t},x^{t-1} y^{m_1},\ldots, xy^{m_{t-1}}, y^{m_t})$, in particular, they are homogeneous. This proves the lemma.
		\end{proof} 
		Now we give a graded structure on the polynomial algebra $\kk[\mathbf{a}]$, the tuple $\mathbf{a}$ was defined in Definition \ref{matrix-general}.
		\begin{defi} \label{grading}
			We endow $R[\textbf{a}]$ with a $\Z^2$-grading
			$$ \deg(a_{i,j})= (i-j,m_j-m_{i-1}),$$
			$$ \deg(a_{i,j,k}) = (i-j,m_j-m_{i-1}-k).$$
			In this manner, we have also induced  a $\G_m^2$-action on the base scheme $B$.
		\end{defi}
		
		\begin{remark} \label{UxUymatrices}
			We define two matrices encoding the bi-degree of $H_I$, $U_x(E)$ and $U_x(E)$ of degrees with respect to the $x$-grading and the $y$-grading, respectively. In \cite[page 5]{Aldo-HB} authors consider the degree matrix $U(E)$ which is exactly $U_x(E)+U_y(E)$.
		\end{remark}
		
		\begin{exm}
			Let us see what are the degree encoding matrices in the case of Example \ref{first-exm}. From Definition \ref{grading} we see that the degrees with respect to the $x$-grading can be read out straight from the dimensions of the matrix $M_E$
			$$U_x(E)=\begin{pmatrix}
				0&-1&-2\\
				1&0&-1\\
				2&1&0\\
				3&2&1\\
			\end{pmatrix}.$$
			We can observe that matrix encoding the $y$-grading is determined by the diagonal $(i,i)$ of the matrix $M_E$
			$$U_y(E)=\begin{pmatrix}
				1&1&3\\
				0&0&2\\
				0&0&2\\
				-2&-2&0\\
			\end{pmatrix}.$$
			
		\end{exm}
	
		\begin{remark}
			To avoid confusion we want to emphasise the degree of a monomial with respect to $x$ is \textbf{not} the same as the degree with respect to the $x$-grading. For example, the monomial $a_{i,j,k}x^l$ has the degree $l$ with respect to $x$, and $i-j+l$ with respect to the $x$-grading.
		\end{remark}
		
	%		\begin{remark} \label{degree-simplified}
	%		Let us observe that directly form definition %$\deg(a_{i,i})=\deg(y^{d_i})$ and $\deg(a_{i+1,i})= \deg(-x)$. Even if in a particular situation some $a_{i,j}$'s are forced to be $0$, while being solely interested in degrees, without loss of generality, we can consider the matrix $\begin{pmatrix}
	%			a_{i,j}
	%		\end{pmatrix}_{1\leq i \leq t, 1\leq j \leq t+1} $.
	%		\end{remark}
 
	 	\begin{lem} \label{homgeneity} 
	 		The morphism $\pi$ described in Definition \ref{matrix-general} is $\G_m^2$-equivariant
	 	\end{lem}
	 	\begin{proof}
			We proceed as in Lemma \ref{simple-resolution}.
			We need to check that the matrix $\mathcal M$ is homogeneous
			 $$(-t+j-1,-m_j-)+(i-j,m_j-m_{i-1})=(-t+i-1,-m_{i-1}),$$
			which is the degree of $i$-th component of the target module. \\
			This implies that the vanishing locus $V(I_t(\mathcal M))$ is a $\G_m^2$-stable subscheme of $\A^2\times B$. Thus the morphism $\pi$, being the projection onto the second factor restricted to the vanishing locus, is {$\G_m^2$-equivariant}.
		 	%	It is enough to prove that maximal minors of matrix $\mathcal M$ are homogeneous with respect to the grading given in Definition \ref{grading}. \vskip 0.2 cm
		 %	 First, we check that the minors are homogeneous with respect $x$-grading, which is the first component of $\Z^2$-grading. The matrix of degrees $U_x$ has form $\begin{pmatrix}i-j \end{pmatrix}_{1\leq i \leq t, 1\leq j \leq t+1}$. From the Leibniz formula for determinants, we get that each maximal minor is homogeneous with respect to $x$-grading.
		 %		 Second, we consider matrix $U_y$ which is of the form 
		 %		 $\begin{pmatrix}m_j-m_{i-1} \end{pmatrix}_{1\leq i \leq t, 1\leq j \leq t+1}$.
		 %		 Again by the Leibniz formula for determinants, maximal minors are homogeneous with respect to $y$-grading.
	 		\end{proof}
 			
 		The lemma above implies that if we take $S\subset B$ a $\G_m^2$-stable subscheme, such that $\pi$ is flat and finite over $S$ then the induced map for $S$ to $\Hilb^d(\A^2)$ is equivariant. This property is crucial in the next section concerned with the tangent space. Also, it is essential to relate our construction with the Białynicki-Birula cells in Section \ref{chapter7}.  
 		
	\section{Tangent space} \label{chapter4}
	In the previous section, we have made our construction equivariant with respect to the $\G_m^2$-action compatible with the usual $\G_m^2$-action on $\A^2$.
	In this section, we use the equivariance to describe how the family of ideals behaves when we infinitesimally perturb the monomial ideal. In Theorem \ref{iso-tangent} we show that any first-order deformation of $E$ is induced by the morphism $\pi$. In Theorem \ref{infinitesimal-phi} we show that this is also the case for higher-order infinitesimal deformations.
	  \vskip 0.2 cm We start by restricting the morphism $\pi$ to the tangent space at the origin $0\in B$. Let us take a first-order infinitesimal thickening of the origin $B_0 \subset B$, such that $T_0B_0 = T_0B$. By the Schaps Theorem (\nolinebreak{Theorem~\ref{Hilbert-Burch}}) morphism $\pi_{|\pi^{-1}(B_0)}$ is flat and by Nakayama's lemma \cite[\href{https://stacks.math.columbia.edu/tag/051F}{Tag 051F}]{stacks-project} it is finite, thus it induces a morphism to $\Hilb^d(\A^2)$. This map induces a map on tangent spaces which we denote by $d\phi_E\colon T_0B \to T_{[E]}\Hilb^d(\A^2)$.

	The map $d\phi_E$ is linear, to show that is is an isomorphism it is enough to check that it is a monomorphism and compute the dimensions of the domain and the codomain. We start by investigating the latter. By \cite{Fogarty} the Hilbert scheme $\Hilb^d(\A^2)$ is smooth of dimension $2d$.
	\begin{exm}
		Let us take the ideal $(x^2,xy^2,y^3)$, it is determined by the sequence fo $y$ exponents on the diagonal of Hilbert-Burch matrix, in this case $(2,1)$.
		One can present the monomial basis of the quotient algebra in the form of a Young diagram.
		$$ \begin{ytableau}
			y^2  \\
			y & xy \\
			1 & x 
		\end{ytableau}$$
		If we add an element $d_3$ as the last element of the sequence of exponents, we will obtain the ideal $(x^3,x^2y^2,xy^3,y^{3+d_3})$.
		The new diagram is constructed from the previous on by shifting it to the right by $1$ and adding a column of height $2+1+d_3$
		
		\begin{center} 
			\begin{tabular}{ c c c }
				\begin{ytableau}
					\\
					\ldots \\
					\ldots \\
					\ldots  \\
					& \\
					& & \\
					& &	  \\ 
				\end{ytableau}
			\end{tabular}
		\end{center}
		
		The same happens in general, thus at each step, we add the sum of the new sequence (in this case $1+2+d_3$) to the length. \vskip 0.2 cm
	\end{exm}
	\begin{lem} \label{tangent-ES} 
		The dimension of $B$ is equal to the dimension of $\Hilb^d(\A^2)$, that is to $2d$.
	\end{lem}
	\begin{proof}
		Let us take a monomial ideal $E$. It is uniquely determined by the sequence of $y$ exponents on the diagonal $\mathbf d = (d_1,\ldots ,d_t)$. We will proceed inductively with respect to $t$, for given $\mathbf d$ the spread-out matrix will be denoted by $\mathcal M_{\mathbf d}$.
		First, let us consider the ideal $E_{d_1}=(y^{d_1},x)$. The colength of $E_{d_1}$ is $d_1$, the quotient algebra has a basis $1,y,y^2,\ldots ,y^{d_1-1}$. The sequence has only one element $(d_1)$. The spread-out matrix $\mathcal M_{(d_1)}$ is
		$$
		\begin{pmatrix}
			y^{d_1}+a_{1,1}\\
			-x+a_{2,1}
		\end{pmatrix}.
	$$
	The dimension of the base scheme $B_{d_1}$ is $2d_1$, each $a_{i,j}$ is a polynomial of dimension lower than $d_1$, thus there are $2d_1$ variables. 
	That concludes the basis for the induction. \vskip 0.2 cm
	Now let us look at the matrix $\mathcal M_{\mathbf d}$, when we add an additional element to the sequence, we are adding one row and one column

		$$	\begin{pNiceMatrix}
		
		y^{d_1} & \ast  & \cdots   & \ast &  	\Block[fill=blue!15,rounded-corners]{6-1}{} \ast \\
		
		\ast & 	y^{d_2}   & \cdots  &  \ast & \ast \\
		\vdots &  \vdots  & \ddots   & \vdots & \vdots  \\
		\ast & \ast  & \cdots  & y^{d_t} & \ast  \\
		\ast & \ast  & \cdots     & \ast & y^{d_{t+1}} \\
		\Block[fill=blue!15,rounded-corners]{1-5}{} \ast & \ast  & \cdots  & \ast  & \ast  \\
	\end{pNiceMatrix}.$$
	
	Each added entry is dominated by exactly one of the diagonal elements $y^{d_j}$, and every such diagonal element dominates exactly two new entries one over and one under the diagonal
$$	\begin{pNiceMatrix}
	
	y^{d_1} & \ast  & \cdots   & \ast &  	\Block[fill=blue!15,rounded-corners]{6-1}{} \ast \\
	
	\ast & 	 \Block[fill=red!30]{1-4}{} \Block[fill=red!30]{5-1}{} y^{d_2}   & \cdots  &  \ast & \ast \\
	\vdots &  \vdots  & \ddots   & \vdots & \vdots  \\
	\ast & \ast  & \cdots  & y^{d_t} & \ast  \\
	\ast & \ast  & \cdots     & \ast & y^{d_{t+1}} \\
	\Block[fill=blue!15,rounded-corners]{1-5}{} \ast & \ast  & \cdots  & \ast  & \ast  \\
\end{pNiceMatrix}.$$
	Thus we have added $2\sum d_j$ parameters.
	\end{proof}
	
	\begin{remark}
		Instead of the computation above we could also almost repeat the argument \cite[Lemma 3.2]{ES1} for Resolution \eqref{gm-res-simple} which is a resolution induced by a slightly different matrix than in \cite[Lemma 3.1]{ES1}.
		% That way we would get tangent space to Hilbert scheme presented in terms of the representation ring of $\G_m^2$. 
	\end{remark}

	\begin{twr}\label{iso-tangent}
		The morphism $d\phi_E$ is an equivariant isomorphism of tangent spaces $T_{0}B~\to T_E~Hilb^d(\A^2)$.
	\end{twr}
		\begin{proof}
			
			In Lemma \ref{tangent-ES} we have shown that the dimension of the domain and the codomain agree, thus it is enough to check that the kernel of $d\phi_E$ is trivial.
			
			Since, by Lemma \ref{homgeneity}, the map $\pi_E$ is equivariant, the induced map of tangent spaces to fixed points is an equivariant linear map between them. It is enough to show injectivity for the subspace of the weight $w_x$ for the fixed weight $w_x$ with respect to the $x$-grading. On the domain it is a linear subspace consisting of $(t+1)\times  t$ matrices "with the third diagonal":
			$$ \mathcal M =
			\begin{pmatrix}
				y^{d_1} & 0 & 0 & \ldots & c_1 & \ldots & \ldots & 0 \\
				-x & y^{d_2} & 0 & \ldots & \ldots & c_2 & \ldots & 0 \\
				0 & -x & y^{d_3} &  \ldots &\ldots & \ldots & c_3 & 0 \\
				\vdots & \vdots & \ddots  & \ddots & \vdots & \vdots & \vdots & \vdots  \\
				0 & 0 & 0 &  \ldots &\ldots & \ldots & \ldots & c_{t-w_x} \\
				\vdots & \vdots & \vdots  & \vdots & \ddots & \ddots & \vdots & \vdots  \\
				\vdots & \vdots & \vdots  & \vdots & \vdots & \ddots & \ddots & \vdots  \\
				0 & \ldots & \ldots  & \ldots & \ldots  & \ldots   & -x & y^{d_t} \\
				0 & \ldots & \ldots & \ldots & \ldots  & \ldots  & 0 & -x \\
			\end{pmatrix},
			$$
			where $c_j$ is on the $(w_x+j,j)$ entry.
			Let us assume that $w_x<0$, so that "the third diagonal" $\{c_i\}$ lies over the diagonal. In the other case, the argument is almost identical. \vskip 0.2 cm
			
			Now we will show that the map of tangent spaces is injective when restricted to such matrices.
			%The $i$-th minor depends only on values of $c_1\ldots c_{i-1}$. This follows from Gauss elimination. We remove $i$-th row, then choose a row with only one non-zero entry and proceed eliminating elements higher than this element.
			Let us assume that the lowest index $i$ such that $c_i\neq 0$ is $i_0$, then the $(i_0+1)$-th minor depends only on $c_{i_0}$. This minor is exactly:
			$$ y^{m_{i_{0}}}x^{t-i_0} - c_{i_0}(-x)^{t-i_0}\Pi_{i<i_0}y^{d_{i}} = y^{m_{i_{0}}}x^{t-i_0} - c_{i_0}(-x)^{t-i_0}y^{m_{i_{0}-1}}. $$
			Since $\deg(c_{i_0})< i_0$, this minor does not belong to $E$, and the thus image of matrix $M$ is not zero, which means that the kernel of $d\phi_E$ is trivial.
		\end{proof}
	We have seen that the map $\pi$ controls all first-order deformations, now we extend this result to completed local rings.
	%\begin{twr} \label{infinitesimal-phi}
	%	The restriction of $B$ to the formal neighbourhood of the origin induces an isomorphism $$\hat d\phi_E\colon \Spec(\kk[[\ab]]) = \Spec(\widehat{\O_{B,0}})\to \Spec(\widehat{\O_{\Hilb^d,E}})  $$
	%	which is a lift of morphism of tangent spaces $d\phi_E$.
	%	\end{twr}
%	\pasttheorem{infinitesimal-phi1}
		 
	\begin{twr}\label{infinitesimal-phi}
		The restriction of $B$ to the formal neighbourhood of the origin induces an isomorphism of spectra of completed local rings $$\hat d\phi_E\colon \Spec(\kk[[\ab]]) = \Spec(\widehat{\O_{B,0}})\to \Spec(\widehat{\O_{\Hilb^d,E}})  $$
		which is a lift of morphism of tangent spaces $d\phi_E$.
	\end{twr}
		\begin{proof}
			Since both $B$ and the Hilbert scheme are smooth of the same dimension, if the map coming from $\pi$ exists then it is an isomorphism and a lift of $d\phi_E$.
			
			We consider an infinitesimal thickening of the origin $0\in B$, let us call it $$(\Spec(A),0) \to (\Spec(B),0),$$
			$$	\begin{tikzcd}
				\V(I_t(M_E)) \arrow[d, "\pi_0"] \arrow[r, hook] & \V(I_t(\mathcal M))_{|A} \arrow[d, "\pi_A"] \arrow[r, hook] & \V(I_t(\mathcal M)) \arrow[d, "\pi"] \\
				\{0\} \arrow[r, hook] & (\Spec(A),0) \arrow[r, hook] & (\Spec(B),0).
				\end{tikzcd} $$
			The map $\pi_A$ is coming from a lift of the matrix $M_E$ to the Artinian algebra $A$, thus by the Hilbert-Burch Theorem \ref{Hilbert-Burch} this map is flat.
			 By the Nakayama's lemma \cite[\href{https://stacks.math.columbia.edu/tag/051F}{Tag 051F}]{stacks-project},
			 this map is also finite of degree equal $\dim_\kk(\kk[x,y]/I_E)$. \vskip 0.2 cm
			 By $\m\subset \kk[\ab]$ we denote the maximal ideal $(a_{i,j,k})_{i,j,k}$.
			 Then we go to the inverse limit to consider the completed local ring. From the local criterion for flatness \cite[Corollary 6.9]{eisenbud}, we again get a flat family.
			 We have a system of $k[[\ab]]$-modules:
			 $$
			 \begin{tikzcd}
			 	\kk^d \arrow[r, "\cong"] & (\kk[\ab]/\m)^d \arrow[r, two heads] & R/E \\
			 	& (\kk[\ab]/\m^2)^d \arrow[u, two heads] \arrow[r, two heads] & R[\ab]/(E+\m^2) \arrow[u, two heads] \\
			 	& \ldots \arrow[r, two heads] \arrow[u, two heads] & \ldots \arrow[u, two heads] \\
			 	{(\widehat{\O_{B,0}}})^d \arrow[r, "\cong"] & (\kk[[\ab]])^d \arrow[r] \arrow[u, two heads] & R[[\ab]]/\hat E, \arrow[u, two heads]
			 \end{tikzcd}
		 	$$
			 both vertical systems satisfy the Mittag-Leffler condition since all vertical arrows are surjective, thus the bottom horizontal arrow is also surjective. As a consequence, $R[[\ab]]/\hat E$ is finite over $\kk[[\ab]]$. This way we have obtained that the morphism $\pi$ over the spectrum of the completed local ring induces a map $\hat d\pi_E$ to the Hilbert scheme. 
		\end{proof}
		From the theorem above, one could speculate that the map $\pi$ induces an \'etale rational map to the Hilbert scheme. In the next section, we will see that in general that is not the case. In Section \ref{chapter6} we establish how to control the failure of the well-definedness of that map. In Section \ref{chapter tangent2} we generalise the above theorem. For any given ideal, not necessarily a monomial one, we give an explicit description of all formal deformations in $\Hilb^d(\A^2)$.
		
	\section{Dream scenario and its failure} \label{chapter5}
		In the previous section, we have considered the infinitesimal behaviour of our family around the origin, in this section we consider the global question.
		\vskip 0.2 cm
		The dream scenario would be if the map $\pi$ was finite and flat, thus induced a map $\phi\colon B \to \Hilb^d(\A^2)$. Then its induced map of tangent spaces at the origin would be $d\phi$, described in the previous section. This way we would end up with a generically \'etale equivariant map to the Hilbert scheme constructed in a very explicit manner. The open subsets obtained that way would contain the Białynicki-Birula cells for any one-dimensional subtorus that has finitely many fixed points on the Hilbert scheme, see Lemma \ref{easy-ABB}. As a consequence, we would construct an interesting collection of explicit $\G_m^2$-stable open subsets giving more information than the Białynicki-Birula cells.
		\vskip 0.2 cm
		The construction would go in the following manner. We take the spread-out matrix $\mathcal M$, see Definition \ref{matrix-general}, it give us a sequence
		$$ R[\ab]^{\oplus  t} \xrightarrow{\mathcal M} R[\ab]^{\oplus  t+1} \to I_t(\mathcal M) \to 0 $$
		just as in the Hilbert-Burch Theorem \cite[Theorem 20.15b]{eisenbud}. If $\mathcal M$ satisfied the condition of the Hilbert-Burch Theorem on every fibre of $\pi$ and $\pi$ was finite, from the local criterion for flatness, we would get that $\pi$ is flat morphism, thus inducing map $\phi$ to the Hilbert scheme.
		The problem is that the dream scenario falls apart, $\mathcal M$ does not need to satisfy the condition of the Hilbert-Burch Theorem. The morphism $\pi$ can admits positive dimensional fibres and even if we restrict it to the locus where it does have only zero-dimensional fibres it doesn't have to be finite.
		\begin{exm} \label{exm-nonflatness}
			The matrix $\mathcal M$ does not need to satisfy the conditions of the Hilbert-Burch Theorem and the morphism $\pi$ does not need to have all fibres zero-dimensional. \vskip 0.2 cm 
			Let us take $E=(y^2,x^2)=(x^2,xy^2,y^2)$.
			In this case, the Hilbert-Burch matrix is
				$$ M_E=
			\begin{pmatrix}
				y^{2} & 0 \\
				-x 	  & 1 \\
				0 & -x 
			\end{pmatrix} .
			$$
			 Instead of working over the full base $B$, let us restrict ourselves to the linear subspace spanned by $a_{1,2,1}$ and $a_{2,1,1}$ and we consider the family of matrices restricted to this subspace:
			 $$ \mathcal M=
			 \begin{pmatrix}
			 	y^{2} & ya_{1,2,1} \\
			 	-x+ya_{2,1,1} 	  & 1 \\
			 	0 & -x 
			 \end{pmatrix} .
			 $$
			 Let us consider the maximal minors of matrix $\mathcal M$
			\begin{align*}
			 f_0 &=    x^2 - xya_{2,1,1}, \\ 
			  f_1 &= -xy^2, \\
			  f_2 &= y^2(1- a_{2,1,1}a_{1,2,1}) + xya_{1,2,1}.  \end{align*}
			 If $a_{2,1,1}a_{1,2,1}=1$ then we get the ideal generated by $(x^2,xy)$. The vanishing locus of such an ideal contains the line $\V(x) \subset \A^2$. Thus not every fibre of $\pi$ is zero-dimensional.
			 This means, by the Hilbert-Burch Theorem \cite[Theorem 20.15a]{eisenbud}, that the sequence
			 	$$ 0 \to R[\ab]^{\oplus  t} \xrightarrow{\mathcal M} R[\ab]^{\oplus  t+1} \to I_t(\mathcal M) \to 0$$
			 	is not exact.
		\end{exm}
		 The existence of positive dimensional fibres is not a major issue by itself, there always is an open neighbourhood of the origin on which the morphism $\pi$ is flat and has zero-dimensional fibres, see Lemma \ref{resultant-flat} below. The next example shows the main problem.
		
		\begin{exm} \label{exm-non-finitness}
				The morphism $\pi$ does not need to be finite over any open neighborhood of the origin. Let us take $E=(y^3,x^2)=(x^2,xy^3,y^3)$.
				The colength of $E$ is $6$, as the quotient algebra has a basis $\{1,x,xy,xy^2,y,y^2\}$ \vskip 0.2 cm
			In this case, the Hilbert-Burch matrix is
			$$ M_E=
			\begin{pmatrix}
				y^{3} & 0 \\
				-x 	  & 1 \\
				0 & -x 
			\end{pmatrix}. 
			$$
			Instead of working over the full base $B$, let us restrict ourselves to the linear subspace spanned by $a_{1,2,2}$ and $a_{2,1,2}$. Then we consider the family of matrices restricted to this subspace
			$$ \mathcal M=
			\begin{pmatrix}
				y^{3} & y^2a_{1,2,2} \\
				-x+y^2a_{2,1,2} 	  & 1 \\
				0 & -x 
			\end{pmatrix}.
			$$
			 Let us consider maximal minors of the matrix $\mathcal M$
			 \begin{align*}
			  f_0 &=    x^2 - xy^2a_{2,1,2},  \\
				f_1 &= -xy^3, \\
			  f_2 &= -y^4a_{2,1,2}a_{1,2,2} + y^3    + xy^2a_{1,2,2}.  \end{align*}
			If $a_{2,1,2},a_{1,2,2} \neq 0$, then the set $\{1,x,xy,xy^2,y,y^2,y^3\}$ is linearly independent in the quotient algebra, thus the colength of the ideal is at least $7$ which is more than it should be. The morphism $\pi$ is flat on some open neighbourhood of the origin $0\in B$, see Lemma \ref{resultant-flat}, thus if it was finite, it would have a constant length of the fibres on some neighbourhood. 
		\end{exm}

	\section{Resultant strategy} \label{chapter6}
		We have seen that the dream scenario falls apart as shown in Examples \ref{exm-nonflatness}, \ref{exm-non-finitness}.
		The strategy to deal with those issues is to create manageable criteria for finiteness and flatness for algebras obtained from the Hilbert-Burch-like matrices, from Definition \ref{matrix-general}. Since the minor that is obtained after removing the first row of $\mathcal M$ has a form $$x^t+ \text{(elements of lower degree with respect to $x$)},$$ to show finiteness of $\pi$ it is enough to show that the element $y$ in $R[\ab]/I_t(\mathcal M)$ is integral over $\kk[\ab]$. That is why we adopt the technique from elimination theory.
		\vskip 0.2 cm
		\begin{defi}\label{resultant}
			Let us consider the matrix $\mathcal M$ from Definition \ref{matrix-general}. The vanishing set of the ideal generated by its maximal minors defines the map $\pi$, see Definition \ref{matrix-general}.
			Let us take the first and last maximal minors of the matrix $\mathcal M$, that is, determinants of $\mathcal M$ after excluding the first and last row respectively.
			We treat those two minors as polynomials of variable $x$ over $\kk[\ab,y]$. We consider their resultant with respect to the variable $x$, it is a polynomial in $\kk[\ab,y]$, and we denote it by $\res_E$.
		\end{defi}
		
		\begin{remark}
			Since the first minor is monic with respect to $x$, by Lemma \ref{res-stable}, after changing the base $\res_E$ continues to be the resultant of two minors of the restricted matrix.
		\end{remark}
		The following lemma resolves the issue exhibited in Example \ref{exm-nonflatness}. 
		%The following lemma and Lemma \ref{lemma-finitness} are two crucial technical ingredients of the whole paper. They answer the issues from Example \ref{exm-nonflatness} and Example \ref{exm-non-finitness}, respectively.
		\begin{lem} \label{resultant-flat}
	 	Let $B'$ be the subscheme of $B$ defined as the complement of vanishing locus of all coefficients of the resultant ${\res_E\in k[\ab][y]}$ as a polynomial of the variable $y$.
		Then $B'$ is a $\G_m^2$-stable open subscheme of $B$. The map $\pi$ is flat over~$B'$.
		\end{lem}
			\begin{proof}
					For the first part of the statement, the resultant of two homogeneous elements with respect to a variable is homogeneous. This implies that the $\G_m^2$-action restricts to $B'$. \vskip 0.2 cm
					For the second part of the statement,
					let $\mathfrak p \subset k[\ab] $ be a maximal ideal corresponding to a closed point of $B'\subset B$, let $\kappa(\mathfrak p)$ be a residual field of $\mathfrak p$.
				 	We check that the ideals $(I_t(\mathcal M))_{\mathfrak p}$ and $I_t(\mathcal(M)) \otimes \kappa(\mathfrak p)$ have the depth equal to $2$.
				 	Let us denote $I=I_t(\mathcal M)$.
				 	Let $f_0,f_t \in R[\ab]$ be the first and the last maximal minor of $\mathcal M$. If they do not have a common factor then they constitute a regular sequence of length $2$ so $\depth (f_0,f_t) =2$, and thus $\depth(I) \geq 2$. Since $\pi$ is a surjective map, $\codim(I)\leq 2$. By \cite[Proposition 18.2.]{eisenbud} $\codim(I) \geq \depth(I)$, implying that $\depth(I)=2$. 
				 	The definition of the resultant says that the resultant $\res_E$ vanishes if and only if $f_0$ and $f_t$ admit a common factor that does not belong to $k[\ab,y]$.
				 	Since $f_0$ is monic with respect to $x$, it does not admit factors that do not depend on $x$. Therefore the resultant vanishes if and only if they do admit a common factor. This implies that the depth of ideals $(I_t(\mathcal M))_{\mathfrak p}$ and $I_t(\mathcal(M)) \otimes \kappa(\mathfrak p)$ over each point of $B'$ is constant and equal to $2$.
				 	
				 	Since the depth of $(I_t(\mathcal M))_{\mathfrak p}$ is equal to $2$, the matrix $\mathcal M$, by the Hilbert-Burch~Theorem~\ref{Hilbert-Burch}, induces the resolution:
					$$ 0\to R[\ab]_{\mathfrak p}^{\oplus  t} \xrightarrow{\mathcal M} R[\ab]_{\mathfrak p}^{\oplus  t+1} \to R[\ab]_{\mathfrak p} \to R[\ab]_{\mathfrak p}/I_t(\mathcal M) \to 0,$$
					by the local criterion of flatness \cite[Theorem 6.8]{eisenbud} it is enough to check that
						$$ 0\to (R[\ab]/{\mathfrak p})^{\oplus  t} \xrightarrow{\mathcal M} (R[\ab]/{\mathfrak p})^{\oplus  t+1} \to R[\ab]/{\mathfrak p} \to R[\ab]/{\mathfrak p}/I_t(\mathcal M) \to 0,$$
					is an exact sequence of $\kappa(\mathfrak p)$-modules, but again since the depth of $I_t(\mathcal(\mathcal M)) \otimes \kappa(\mathfrak p)$ is $2$, this sequence satisfies the assumptions of Hilbert-Burch Theorem ~\ref{Hilbert-Burch}.
			\end{proof}
			To understand finiteness of the map $\pi$ we first need to see what is the resultant over the origin.
			\begin{exm} \label{res-origin}
					On the origin of $B$, the first minor $f_0$ is $x^t$ and the last one is $y^{m_t}$, the resultant $\res_x(x^t,y^{m_t})$ is equal to $y^{tm_t}$ as it is a determinant of the $t\times t$ Sylvester matrix
				$$	\begin{pmatrix}
					y^{m_t} & 0 & 0 &0 & \ldots & \ldots \\
					0 & y^{m_t} & 0 &0 & \ldots  & \ldots\\
					0 & 0 & y^{m_t} &0 & \ldots  & \ldots\\
					0 & 0 & 0 & y^{m_t} & \ldots  & \ldots\\

					\vdots & \vdots & \vdots & \vdots & \ddots & \vdots \\
					\ldots & 0 & 0 & 0 & \ldots & y^{m_t}  \\
					\end{pmatrix}. $$
			\end{exm}
			Now we state the crucial corollary of this paper.
			\begin{stwr} \label{lemma-finitness}
				On the locally closed subscheme $S \subset B'$ on which the resultant is a polynomial of degree exactly $tm_t$, the map $\pi$ is finite and flat.
			\end{stwr}
				\begin{proof}
				As the statement is local on the target we substitute $S$ with an affine open subscheme.
				As we have seen in the computation above, Example \ref{res-origin}, the resultant $\res_E$ is equal to $y^{tm_t}$ modulo the ideal generated by parameters $(a_{i,j})$. In particular, $S$ is non-empty. On $S$ the resultant is of the form
				$$uy^{tm_t}+\text{terms with lower degrees with respect to } y$$
				where $u$ is invertible on $S$. Since $\res_{E} \in I_t(\mathcal M)$, this relation makes $y \in H^0(S,\O_S)[x,y]/I_t(\mathcal M)$ integral over $H^0(S,\O_S)$. As the maximal minor $f_0$ is monic with respect to $x$ it makes $x$ integral over $H^0(S,\O_S)[y]$, and thus we obtain that $H^0(S,\O_S)[x,y]/I_t(\mathcal M)$ is finite over $H^0(S,\O_S)$.
				Since the resultant has degree exactly $tm_t$ it is not zero independently on the values of parameters $a_{i,j,k}$, thus by Lemma \ref{resultant-flat} we get flatness.
				\end{proof}
			\begin{remark}
				In general, the subscheme $S$ is not closed nor open in $B$ since it is given by:
				\begin{itemize}
					\item One open condition that the coefficient with $y^{tm_t}$ does not vanish.
					\item Closed conditions that coefficients with higher exponents of $y$ are zero.
				\end{itemize}
	
			\end{remark}
			To see how the resultant strategy works in practice, let us look at Example~\ref{exm-nonflatness} and Example~\ref{exm-non-finitness} using this tool.
		\begin{exm} \label{exm-resultant}
			Let us continue the investigation started in Example \ref{exm-nonflatness}, and apply Corollary \ref{lemma-finitness}.
			We get $E=(y^2,x^2)$, and want to compute the resultant of $x^2 - xya_{2,1,1} $ and $xya_{1,2,1} + y^2(1- a_{2,1,1}a_{1,2,1})$ with respect to the variable $x$. Let us denote $a_{1,2,1}$ by $a$ and $a_{2,1,1}$ by $b$.
		    To compute the resultant is take the Sylvester matrix
	$$		\begin{pmatrix}
				1 & -yb & 0 \\
				ya & y^2(1- ba) & 0\\
				0 &  ya & y^2(1- ba) \\
			\end{pmatrix},
			$$
			and compute that the determinant of whose is \[y^4(1-ba)^2-y^4(-ba)(1-ba) 
			 = y^4(1-ba).\]
			We see the resultant is $0$ if and only if $ab=1$, which is exactly when $\mathcal M$ admits a non-trivial kernel in Example~\ref{exm-nonflatness}.
			Thus we see that in the case of the ideal $(x^2,y^2)$, we got an induced rational map to the Hilbert scheme.
			
		\end{exm}
	
			\begin{exm} \label{exm-resultant2}
		Let us see how this technique works out in the case presented in Example \ref{exm-non-finitness}.
		We get $E=(y^3,x^2)$, and want to compute the resultant of $x^2 - xy^2a_{2,1,2} $ and $     xy^2a_{1,2,2} - y^4a_{2,1,2}a_{1,2,2} + y^3$ with respect to the variable $x$. Let us denote $a_{1,2,2}$ by $a$ and $a_{2,1,2}$ by $b$.
		To compute the resultant is to take the Sylvester matrix
		$$		\begin{pmatrix}
			1 & -y^2b & 0 \\
			y^2a & y^3(1- yba) & 	0\\
			0 &  y^2a & y^3(1- yba) \\
		\end{pmatrix},
		$$
		and compute that the determinant of whose is 
		$$ y^7(ab) -y^6. $$
		In Example~\ref{exm-non-finitness}, we have seen that over $ab\neq 0$ the length of the fibres is higher than over the origin. Corollary \ref{lemma-finitness} tells us that on the closed locus $ab= 0$ we again get a map to the Hilbert scheme of points.
		
	\end{exm}

	\section{Białynicki-Birula decomposition}\label{chapter7}
		In the previous section, we have developed a tool that allows us to control the finiteness and flatness of the map $\pi$. In this section, we use this tool to show that our family describes the Białynicki-Birula cells of the Hilbert scheme of points. Since the nature of the Białynicki-Birula cells is equivariant, we heavily use the equivariance of all of the previous results. We start with a general discussion about the Białynicki-Birula decomposition on the Hilbert scheme.
		\vskip 0.2 cm
		We consider the affine plane $\A^2$ with the standard $\G_m^2$-action
		$$ (\lambda_1,\lambda_2) \cdot (x,y) = (\lambda_1 x,\lambda_2 y),$$
		 This action induces one on $\Hilb^d(\A^2)$, and for a choice of monomial ideal $E$ also on the base scheme $B_E$, see Definition~\ref{matrix-general}. Choosing a nonzero cocharacter:
		 \begin{align*}
		 		 \psi\colon \G_m &\to \G_m^2\\
		 		 t &\mapsto (t^\a,t^\beta)
		 \end{align*}
		is equivalent to endowing $R=\kk[x,y]$ with $\Z$-grading, denoted by $\deg_\psi$, such that $\deg_\psi(x)=\a$ and $\deg_\psi(y)=\b$.
		Such a choice of $\psi$ induces the Białynicki-Birula decomposition on $B$ and $\Hilb^d(\A^2)$. For specific reference for the Hilbert scheme see \cite{ES1}. For general reference on the Białynicki-Birula decomposition see \cite[Theorem 1.5]{ABB_JJLS}. The decomposition associated to $\psi$ is denoted by adding superscript $+,\psi$, e.g. $\Hilb(\A^2)^{d,+,\psi}$.
		Note that the only fixed points of $\G_m^2$-action on $\Hilb^d(\A^2)$ are the monomial ideals and for $B$ the only $\G_m^2$-fixed point is the origin. Thus for any choice of $\psi$ those subschemes are contained in the fixed point locus of $\psi(\G_m)$ acting respectively on the Hilbert scheme and $B$.
		\vskip 0.2 cm
		\begin{exm}
			We choose a cocharacter $\psi$ such that $\psi(\G_m)$ admits only finitely many fixed points.
			We describe the $\G_m$-action as follows:
			\begin{align}
				\G_m \times \Hilb^d(\A^2)&\to \Hilb^d(\A^2),\\
				(t,[I]) &\mapsto (\psi(t)\cdot [I]),
			\end{align}
			Then the Białynicki Birula cell associated to $\psi$ and centred at $[E]$ consists if the ideals $[I] \in \Hilb^d(\A^2)$ such that \[ \lim_{t\to 0} t\cdot [I] = [E]. \]
			
		\end{exm}

		 In general, the subtorus given by the cocharacter $\psi$ has more fixed points.
		\begin{defi} \label{ABB-cell-associ}
			 The Białynicki Birula cell of $\Hilb^d(\A^2)$  (respectively, $B$) associated to the cocharacter $\psi$ and centred at the monomial ideal $E$ is the connected component of the smooth variety $\Hilb^{d,+,\psi}(\A^2)$ (respectively  $B^{+,\psi}$) that contains the point $[E]$ (respectively, the origin).
		 \end{defi}
		Since $B$ is an affine space with a linear $\G_m^2$-action, it is relatively easy to establish what is $B^{+,\psi}$. It is an affine subspace spanned by those variables $a_{i,j,k}$ with bi-degree $(w_x,w_y)$ such that the induced degree $\a w_x+\beta w_y$ is non-negative, see \cite[Proposition 2.15]{ABB_JJLS}. 

		In other words, $\psi$ induces a functional on the space of all characters of $\G_m^2$. For any chosen parameter $a_{i,j,k}$, the line $\langle {a_{i,j,k}} \rangle$ can be treated as a character of $\G_m^2$, and then $B^{+,\psi}$ is spanned by those characters that are positive with respect to the cocharacter $\psi$. 
		\vskip 0.2 cm
		\begin{exm}
			Let us take an ideal $E=(x^2,xy,y^2)$.
			The spread-out matrix $\mathcal M$ is
			$$\begin{pmatrix}
				y+a_{1,1} & a_{1,2}\\
				-x+a_{2,1} & y+a_{2,2} \\
				a_{3,1}  & -x+a_{3,2}
			\end{pmatrix},$$
			where for simplicity of notation $a_{i,j}=a_{i,j,0}$.
		Let us plot the weights of variables $a_{i,j}$ on the plane
		$$	\includegraphics[scale=0.5]{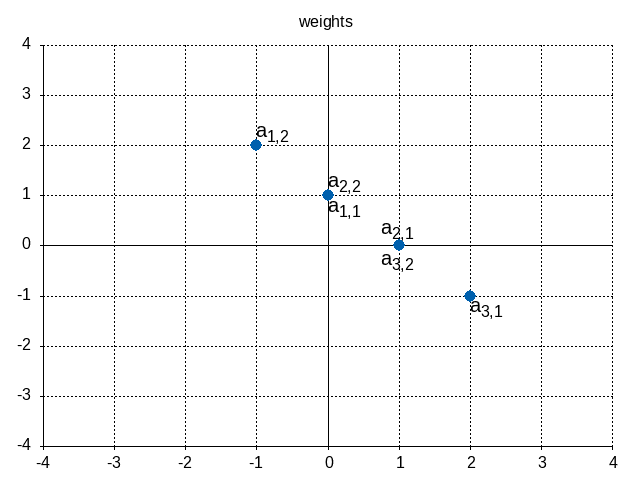} $$
		We choose a cocharacter $\psi\colon t \to (t^{1},t^{-1})$. This induces a functional on the space of weights, with the kernel being a line $x=y$
			$$	\includegraphics[scale=0.5]{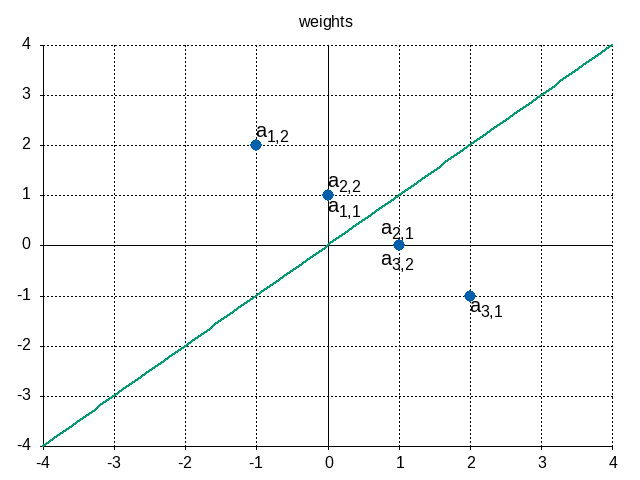} $$
		The parameters which weights lie in the upper half-plane span the Białynicki-Birula cell associated to the cocharacter $\psi$.
		\end{exm}
		Let us take a one-parameter subgroup $\psi$ acting on $\A^2$, that attracts the line $x=0$ to the origin.
		Now we show that the spread-out matrix, Definition \ref{matrix-general}, provides a rational map to the cell on the Hilbert scheme associated to such a $\psi$.
	\begin{twr}\label{theorem-abb}
		Choose a cocharacter $\psi\colon \G_m \to \G_m^2$, such that $y$ has a non-negative degree with respect to the induced grading.
		Let us restrict $B$ to $B^{+,\psi}$ the Białynicki-Birula cell associated to $\psi$. The morphism $\pi$ induces a rational $\G_m^2$-equivariant map $\phi^+ \colon B^{+,\psi} \to \Hilb^{d}(\A^2)$. The image of $\phi^+$ lies in the Białynicki-Birula cell associated to $\psi$ and centred at $E$. The map $\phi^+$ is \'etale in a $\G_m^2$-stable neighbourhood of the origin.
	\end{twr}
	%	\begin{twr} \label{theorem-abb}
	%		Let us choose a cocharacter $\psi\colon \G_m \to \G_m^2$, such that $y$ has a non-negative degree with respect to the induced grading.
	%		Let us restrict $B$ to $B^{+,\psi}$ the Białynicki-Birula cell associated to $\psi$. The $\pi$ induces rational map $\phi^+ \colon B^{+,\psi} \to Hilb^{d}(\A^2)$. The image of $\phi^+$ lands in the Białynicki-Birula cell associated with $E$ and $\psi$.
	%	\end{twr}F
			\begin{proof}
		%		First let us assume that $\deg_{\psi}(y)\neq 0$.
		 	As observed in Example \ref{res-origin} the resultant $\res_E$ is of degree $\deg_\psi(y)tm_t$. We need to show that on $B^+$ the resultant has no terms with degree higher than that with respect to $y$. As noted in Lemma \ref{resultant-flat} the resultant is homogeneous. If there is a non-zero term $W(\ab)y^{N}$ for $N>tm_t$, then $\deg_\psi(W(\ab))=\deg_\psi(y)(tm_t-N)<0$. Since $B^+$ is the Białynicki-Birula cell, $W(\ab)$ depends only on parameters with non-negative degrees with respect to $\psi$, thus $\deg_{\psi}(W(\ab)) \geq 0$.
		 			 	\vskip 0.2 cm
		 	Now, let us assume that $\deg_{\psi}(y)= 0$.
		 	That means that $\psi$ has a form:
		 	$$ t \to (t^\a,1),$$
		 	for $\a\neq 0$.
		 	Again we consider $W(\ab)$ as before, since $\res_E$ is homogeneous with respect to the full $\Z^2$-grading, $W(\ab)$ is also homogeneous and has degree $(0,tm_t-N)$.
		 	Let us consider what parameters can $W(\ab)$ depend on. Since we restrict ourselves to the Białynicki-Birula cell, it can only depend on those parameters $a_{i,j,k}$ such that their degree $(w_x,w_y):=(i-j,m_j-m_{i-1})$ satisfy the following inequality:
		 	$$ \a w_x + 0 \cdot w_y \geq 0.$$
		 	The only way that one can obtain $0$ as a sum of non-negative/non-positive numbers is by taking a sum of $0$-s.
		 	Thus $W(\ab)$ depend only on parameters lying on the main diagonal $(i,i)$.
		 	 As a consequence, we only need to study the resultant coming from the matrix with non-zero parameters only on the main diagonal, $a_{i,j,k}=0$ for $i\neq j$. In that case, the first minor is just $x^t$, and the last one is a product of elements on the diagonal, in particular, it is a polynomial in $y$ of degree $m_t$. \vskip 0.2 cm

		 	Therefore in both cases above, the resultant is the last minor to the power of $t$, so has the predicted degree. By Lemma~\ref{resultant-flat} we get that $\pi$ is flat over an open neighbourhood of the origin, and by Lemma \ref{lemma-finitness} we get that the family $\pi$ is finite. 
		 	\vskip 0.2 cm
		 	The last part follows from Theorem \ref{infinitesimal-phi}.
			\end{proof}

			\begin{remark}
			Since the map $\phi^+$ in Theorem \ref{theorem-abb} above is \'etale and equivariant, it restricts to \'etale map between the $\G_m$-fixed point loci.
			In particular the dimension of the connected $\G_m$-fixed point component of $\Hilb^d(\A^2)$ containing $[E]$ is equal to the dimension of the connected fixed point component of $B$ through the origin.
		\end{remark}
	
		The only $\G_m^2$-fixed points on the Hilbert scheme are the monomial ideals, thus for a general choice of the one-dimensional subtorus $\G_m$, they are the only fixed points. In this special case, we will get that our map is an isomorphism onto the cell.
	%	Before going to this special case of only isolated fixed points, we prove a technical lemma.

		\begin{lem} \label{easy-ABB}
			Let us take the setup from Theorem \ref{theorem-abb}, but also let us assume that the only $\psi(\G_m)$-fixed point on $B^+$ is the origin. Then the induced map $\phi^{+}$ is an isomorphism onto the Białynicki-Birula cell.
		\end{lem}
			\begin{proof}
				Let us consider the coefficient $W$ with $y^{tm_t}$ in $\res_E$. Since the resultant is homogeneous we have $\deg_{\psi} W=0$. The only fixed point is the origin, therefore $\deg_\psi(a_{i,j,k})>0$ for every parameter appearing in $B^+$. As a consequence, $W=1$ and we get that $\pi$ is flat over whole $B^+$. 
				By the proof of Theorem \ref{theorem-abb}, we get that globally the resultant has the degree at most $tm_t$, thus we obtain a map $\phi \colon B^+ \to \Hilb^d(\A^2)$ to the Hilbert scheme, globally.
				\vskip 0.2 cm
				The morphism $\phi$ is a map between the Białynicki-Birula cells (with single fixed points) that is an isomorphism on tangent spaces, Theorem \ref{iso-tangent}. By Lemma \ref{ABB-simples}, the map $\phi$ is an isomorphism.
			%	Let us take a point $b\in B^+$, action of $\A^1$ induces filtration on $H^0(\pi^{-1}(b),\O_{\pi^{-1}(b)})$ such that the associated graded algebra are exactly sections of fibre over the limit $\lim_{t\to 0} t\cdot b = 0 \in B^+$ the limit, which is algebra over the origin $R/E$. In particular length $H^0(\pi^{-1}(b),\O_{\pi^{-1}(b)})$ is constant equal to $d$. Thus the map $\pi$ is finite, and together with the previous part it is also flat, thus we get a regular map $\phi^+$. 
			\end{proof}
			\begin{remark}
			Lemma \ref{easy-ABB} has been essentially proven before in \cite[Lemma 1]{ES-2fix}. In that paper, the matrices considered were slightly different, and at least in principle, the proof there assumed that $\kk=\C$, by considering Euclidean neighbourhoods.
			\end{remark}
		
		We have already seen that in the case of a positive-dimensional fixed point locus, the induced map does not need to be regular, see Example \ref{exm-resultant}. One can ask if it is injective, the following example provides a negative answer.
		\begin{exm} \label{exm-noninjectivness}
			Let us consider the ideal $E=(y^{13},xy^7,x^2y^3,x^3)$.
			Consider a two-dimensional family given by the following matrix
			$$
		\mathcal M= 	\begin{pmatrix}
				 y^3  & a_{1,2,2}y^2 & 0        \\
				 -x & y^4   &     0           \\
				 0  &-x   &     y^6          \\
				0  & 0    &     a_{4,3,5}y^5-x
			\end{pmatrix}.
			$$
			We denote $a_{1,2,2}$ by $a$ and $a_{4,3,5}$ by $b$.
		%	Recall that the bi-degree of $a_{i,j,k}$ is $$(i-j,m_{j}-m_{i-1}-k)$$ and $$m_{j}=\sum_{l\leq j} d_l,$$ where $d_l$ are the exponents of $y$ on the diagonal.
		%	 Let us observe that the parameters $a,b$ have opposite weights ($(-1,5)$ and $(1,-5)$ respectively) with respect to the $\G_m^2$-action,. 
			 Let us compute the degrees of $a,b$ with respect to the bi-degree. Since the spread-out matrix itself is homogeneous, see Lemma \ref{homgeneity}, then the top-left $2\times 2$ minor is also homogeneous, implying that the bi-degree of $a$ is $(-1,5)$. Also the entries of the matrix are homogeneous, which shows that $b$ has the bi-degree $(1,-5)$. The weights are opposite.
			 In particular, this family belongs to the Białynicki-Birula cell with a positive dimensional fixed locus.\\
			The ideal $I:=I_t(\mathcal M)$ is 	$${}\left(y^{13}+x\,y^{8}a,\,y^{12}b+x\,y^{7}ab-x\,y^{7}-x^{2}y^{2}a,\,-x\,y^{8}b+x^{2}y^{3},\,x^{2}y^{5}b-x^{3}\right)$$
			To simplify even further, let us assume that $a=b$ and $a^2=1$. 
			Since the coefficient with $xy^7$ in the second generator is $ab-1$, now this term vanishes. As a consequence, the second generator is after those simplifications equal to 
			$y^{12}a-x^{2}y^{2}a$.
			Since we have already assumed that $a\neq 0$, we may multiply this minor by $a^{-1}y$ and then subtract it from the first one, obtaining:
			$$ axy^8+x^2y^3,$$
			then by adding this to the third generator we obtain
			$2x^2y^3$, if the characteristic of field is not $2$ then, that means that our ideal contains $x^2y^3$. Thus also $xy^8$ and $x^3$ are contained in the ideal. As a consequence, also $y^{13}$ is contained.
			All in all, we get the ideal generated by:
				\begin{equation} \label{exm-ideal}
					( y^{13},y^{12}-x^2y^2,x^2y^3,x^3 ),
			 	\end{equation}
			but our assumption on the parameters has two solutions $a=1,b=1$ and $a=-1,b=-1$, thus for those two choices of the parameters we get the same ideal. If the characteristic of the field was $2$, then those two choices of parameters are identified anyway. \vskip 0.2 cm
			The presentation given in \eqref{exm-ideal} suggests that this ideal belongs to a cell associated to a monomial ideal ${E'=(y^{12},xy^{12},x^2y^3,x^3)}$. 
			Since considered ideals lie in the fixed component of the torus $(t^1,t^5)$ it should be the case also on this level, so we consider the fixed locus of the base scheme $B$ associated to $E'$
			$$\begin{pmatrix}
				y^{3}&0&y^{2}a_{1,3,2}\\
				-x&y^{9}&0\\
				0&-x&1\\
				0&0&-x\\
			\end{pmatrix}.$$
			The ideal of maximal minors is 
			${}\left(y^{12}+x^{2}y^{2}a_{1,3,2},\,-x\,y^{12},\,x^{2}y^{3},\,-x^{3}\right)$
			which for $a_{1,3,2}=1$ gives exactly ideal \eqref{exm-ideal}. Even more than that, since everything is equivariant, this shows us that on the orbit $ab=1$, we have the same equality. To summarize we get the orbit $ab=1$ in $B_E$ that is mapped non-injectively to the Hilbert scheme, but on the Hilbert scheme, this orbit approaches a different monomial ideal $E'$.
		
		\end{exm}
		%	\begin{remark}
		%	%	Gr\"obner order
		%	\end{remark}
		\vskip 0.2 cm
		In Theorem \ref{theorem-abb} and Lemma \ref{easy-ABB} we have assumed that the induced degree of $y$ is non-negative, now we look at the remaining cases. We take a cocharacter $\psi$ such that $\deg_\psi(y)<0$ and an ideal that, when acted by the one-dimensional torus, converges to a monomial ideal when elements of the torus approach $0$. This ideal has to be supported on the line $y=0$, otherwise, the part supported outside of this line diverges to infinity.
		Unfortunately, if we take $B$ and restrict it to the Białynicki-Birula cell with respect to such a cocharacter, some of the fibres of $\pi$ can be supported outside of the kine $y=0$, as illustrated by the following example.
	
		 \begin{exm}
		 %%	$seq=(1,0,0,2)$ weights as in \cite{RoserWinz-2}
		 	Let us consider the monomial ideal $E=(y^3,xy,x^4)$. Let us take the 
		 	cocharacter  $\psi\colon t \to (t^{-2},t^{-3})$.
		 	By Definition \ref{matrix-general}, parameters of the spread-out matrix $\mathcal M_E$ appear only in the top-left and bottom-right corner. Now take the formula for the bi-degree of each parameter 
		 	$$ \deg(a_{i,j,k})=(i-j,m_j-m_{i-1}-k),$$
		 	$$m_{j}=\sum_{l\leq j} d_l,$$ where $d_l,$ are the exponents of $y$ on the diagonal.
		 	We establish that the following inequality on the degree with respect to $\psi$ holds 
		 	$$\deg_{\psi}(a_{i,j,k}) = -2(i-j)+(-3)(m_j-m_{i-1}-k) \geq 0,$$ 
		 	only for indices $(1,3,0)$ and $(5,4,1)$.
		 	Thus when we restrict the spread-out matrix to the cell associated to $\psi$ we get the matrix:
		 	$$\mathcal M = \begin{pmatrix}
		 		y&0&a_{1,3,0}&0\\
		 		-x&1&0&0\\
		 		0&-x&1&0\\
		 		0&0&-x&y^{2}\\
		 		0&0&0&y\,a_{5,4,1}-x\\
		 	\end{pmatrix}.$$
	 		For simplicity let us denote $a_{1,3,0}$ by $a$ and $a_{5,4,1}$ as $b$. 
		 	The first minor is	$x^{4}-x^3yb$ and the last is equal to  $y^{3}+x^{2}y^{2}a$.
		 	We compute the resultant by taking the determinant of the Sylvester matrix
		 		$$\begin{pmatrix}
		 			\vphantom{\left\{-6\right\}}1&-y\,b&0&0&0&0\\
		 			\vphantom{\left\{-5\right\}}0&1&-y\,b&0&0&0\\
		 			\vphantom{\left\{-8\right\}}y^{2}a&0&y^{3}&0&0&0\\
		 			\vphantom{\left\{-7\right\}}0&y^{2}a&0&y^{3}&0&0\\
		 			\vphantom{\left\{-6\right\}}0&0&y^{2}a&0&y^{3}&0\\
		 			\vphantom{\left\{-5\right\}}0&0&0&y^{2}a&0&y^{3}\\
		 		\end{pmatrix}.$$
		 		The resultant is $y^{13}ab^2+y^{12}$, it is not monic, so it does not guarantee finitness of the morphism $\pi$. If $ab^2=1$, then for $y=-1$ the resultant vanishes.
		 		Let us take $a=b=1$, then the ideal of maximal minors is
		 		$$I_{1,1}=(x^{2}y^{2}+y^{3},\,-x^{3}+x^{2}y-x\,y+y^{2},\,x^{2}y-x\,y^{2},\,-x^{3}y+x^{2}y^{2},\,x^{4}-x^{3}y).$$
		 		Let us set $y$ to be $-1$ then we obtain
		 		$$(x^{2}-1,\,-x^{3}-x^2\,+x+1,\,-x^{2}-x,\,x^{3}+x^{2},\,x^{4}+x^{3})$$
		 		which is contained in the ideal $(x+1)$. 
		 		We see that the fibre over $a=b=1$ is not contained in any infinitesimal thickening of the line $y=0$, thus it does not describe the Białynicki-Birula cell of the Hilbert scheme for $\psi$.
		 		\vskip 0.2 cm
		 		In general, the primary decomposition of the ideal $I_4(\mathcal M)$ of the maximal minors is
		 		$$(y^{3},\,x\,y^{2},\,x^{2}y,\,x^{3}a-y^{2}b+x\,y,\,x^{4}) \cap  (y\,b-x,\,x\,a\,b+1,\,x^{2}a+y).$$
		 		We see that those two ideals are coprime; one contains $x^4$ and the second one $ xab+1$. Thus the primary decomposition gives two connected components of $\V(I_4(\mathcal M))$, the one that contains the whole fibre over the origin describes a cell on the Hilbert scheme.
		 \end{exm}
	 
		To fix this problem, from the vanishing locus of $I_t(\mathcal M)$ we cut the locus where $y^{tm_t}$ vanishes. On each fibre it means that we intersect the finite scheme with the thickened line $y^{tm_t}=0$. This way we obtain a closed subscheme of the previous one, finite schemes obtained this way are set-theoretically supported on $y=0$. The restriction of $\pi$ to this closed subscheme we denote by~$\pi_{-}$
						\begin{figure}[H] $$\begin{tikzcd}
							\V(I_t(\mathcal M)+(y^{tm_t})) \arrow[r, "cl", hook] \arrow[rrd, "\pi_-"'] & \V(I_t(\mathcal M)) \arrow[r, "cl", hook] \arrow[rd, "\pi"] & \A^2 \times B \arrow[d] \\
							&                                                             & B.                      
						\end{tikzcd}$$ \caption{Family supported on $x$-axis}
					\label{thicken-x-axis}
				\end{figure}
		 The morphism $\pi_-$ is finite since $y$ is forced to be integral, but there is no reason, at least a priori, that $\pi_-$ is again flat.
		\begin{twr}\label{theorem-abb-}
			Let us choose a cocharacter $\psi\colon \G_m \to \G_m^2$, such that $y$ has a negative degree with respect to the induced grading.
			Let us restrict $B$ to $B^{+,\psi}$, the Białynicki-Birula cell associated to $\psi$. Then $\pi_{-}$ induces a rational $\G_m^2$-equivariant map $\phi^+ \colon B^{+,\psi} \to \Hilb^{d}(\A^2)$. The image of $\phi^+$ lies in the Białynicki-Birula cell associated to $\psi$ and centred at $E$. The map $\phi^+$ is \'etale in  a $\G_m^2$-stable neighbourhood of the origin.
		\end{twr}
		\begin{proof}
			We imitate the proof of Theorem \ref{theorem-abb}.
			Let us first consider the morphism $\pi$ on $B^{+}$, and denote the ideal $I_t(\mathcal M)$ by $I$. The resultant $\res_E$ (on $B^+$) is homogeneous of degree $\deg_{\psi}(y)(tm_t)$, since $\deg(y)<0$, it is a negative number. Let $u \in \kk[\ab]$ be a coefficient with $y^{tm_t}$, then
			$$ \res_E = uy^{tm_t}+y^{tm_t+1}W(\ab,y) = y^{tm_t}(u+yW(\ab,y))$$
			Let us restrict ourselves to the open locus of $B^+$ where $u$ is invertible.
			We want to argue that the subscheme containing those points of $\V(I)$ that are not supported on $y=0$ is closed. 
			Let us observe that $$(I+(y^{tm_t})) \ \cap \ (I+(u+yW(\ab,y))) = I,$$
			that is because the principal ideals $(y^{tm_t})$ and $(u+yW(\ab,y))$ are coprime.
			The ideal $I+(u+yW(\ab,y))$ cuts exactly part of $\V(I)$ that is not supported on $y=0$. Thus its complement $\V(I+(y^{tm_t}))$ is open in $\V(I)$, therefore $\pi_-$ is flat over open locus of $B^+$ where $u$ is invertible. 
			By adding $y^{tm_t}$ to the ideal we have enforced finiteness and as a consequence get the induced rational map $\phi^+$
			\vskip 0.2 cm
			The last part follows from Theorem \ref{infinitesimal-phi}.
		\end{proof}
		In the case of $\deg_{\psi}(y)<0$ we have also an analogue of Lemma \ref{easy-ABB}.
		\begin{stwr} \label{easy-ABB-}
		Let us take the setup from Theorem \ref{theorem-abb-} but also let us 
		assume that the only $\psi(\G_m)$-fixed point on $B^+$ is the origin. Then the induced map $\phi^{+}$ is an isomorphism onto the Białynicki-Birula cell.
		\end{stwr}
			\begin{proof}
				Analogous to the proof of Lemma \ref{easy-ABB}.
			\end{proof}
		\vskip 0.2 cm	
	
		Now we will apply our results to answer the question asked in \cite{RoserWinz-2}.\\
		Let $E=(x^{t},x^{t-1} y^{m_1},\ldots, xy^{m_{t-1}}, y^{m_t})\subset \kk[[x,y]]$ be a monomial ideal of finite colength, recall that ${d_i=m_i-m_{i-1}}$. 
		Let $\mathcal N_{ \leq \underline d}$ denote the affine space consisting of matrices $(n_{i,j})$ of size $(t+1)\times t$ with entries in  $\kk[y]$, such that

			\[\begin{cases}	 \label{RH-conditions}
					u_{i,j} &< \ord(n_{i,j}) < d_i\text{ , } i\leq j, \\
					u_{i,j} &\leq \ord(n_{i,j}) < d_j\text{ , } i> j, 
			\end{cases} \]
		where the matrix $U=(u_{i,j})$ is the total degree matrix $U(E)$ defined in Remark \ref{UxUymatrices}.
		The matrix $H$ is the Hilbert-Burch matrix $M_E$ from Definition \ref{matrix-general}. The scheme $\mathbf V(E)$ is the Gr\"obner cell of $E$ with respect to the negative degree lexicographical ordering, see Example \ref{exm-glex}.
		\begin{stwr} [{\cite[Conjecture 4.2]{RoserWinz-2}}]
			Let $E\subset \kk[[x,y]]$ be a monomial ideal of finite colength. Then the map
			\begin{align*}
				 \Phi_e \colon \mathcal N_{<\underline d} &\to \mathbf V(E) \\
				 N &\mapsto I_t(H+N)
			\end{align*}
		is an isomorphism.
		\end{stwr}
		\begin{proof}
			We are following the dictionaries from Tables \ref{table} and \ref{table2}. According to Example \ref{exm-glex}, $\mathbf V(E)$ is the Białynicki-Birula cell for the following grading:
			$$ \begin{cases}\deg(y)=-N \\
			 \deg(x)= -N(1-\varepsilon), \end{cases}$$
		 for $\epsilon \in \Q_{\geq 0}$ and $\epsilon N\in \Z_{\geq 0}$, chosen for $E$.
		 We can observe that $N_{\leq d}+H$ is the spread-out matrix $\mathcal M_E$ restricted to the Białynicki-Birula cell associated to the grading above. Thus the map $\Phi_E$ is the same map as $\phi^+$ considered in Corollary \ref{easy-ABB-}, thus making the statement the special case of Corollary \ref{easy-ABB-}.
		\end{proof}
	\section{Infinitesimal deformations} \label{chapter tangent2}
		In this section we use the results obtained in the previous sections to generalise Theorem \ref{infinitesimal-phi} from monomial ideals to any ideal. \vskip 0.2 cm
		The idea is the following. We know that the family controls higher-order infinitesimal deformations of the monomial ideal. Since every ideal can be deformed to a monomial one by going to the limit with $\G_m$-action, we want to translate the isomorphism of tangent spaces by the group action. Then we repeat the argument from the proof Theorem \ref{infinitesimal-phi} saying that by Schaps Theorem \ref{Hilbert-Burch} we get an isomorphism of completed local rings.
		\begin{twr} \label{infinitesimal-phi2} 
			Let us take an ideal $[I]\in \Hilb^d(\A^2)$. Every monomial ideal $E$ belonging to the closure of $\G_m^2$-orbit of $[I]$ induces an isomorphism of spectra of complete local rings.
			$$ \widehat{d\phi_{I,E}} \colon \Spec(\kk[\ab]) = \Spec(\widehat{\O_{B_E,b}})\to \Spec(\widehat{\O_{\Hilb^d,I}}),$$
			where $b\in B_E$ is a point such that the ideal of maximal minors of the matrix $\mathcal M_E$ over $B_E$ is equal to $I$.
		\end{twr}
	%	\begin{twr}\label{infinitesimal-phi2}
		%	Let us take an ideal $[I]\in \Hilb^d(\A^2)$. Every monomial ideal $E$ belonging to the closure of $\G_m^2$-orbit of $[I]$ induces an %isomorphism of complete local rings.
	%		$$ \widehat{d\phi_{I,E}} \colon \Spec(\kk[\ab]) = %\Spec(\widehat{\O_{B_E,b}})\to \Spec(\widehat{\O_{\Hilb^d,I}}),$$
	%		where $b\in B_E$ is a point such that the ideal of maximal minors of the matrix $\mathcal M_E$ over $B$ is equal to $I$.
	%	\end{twr}
			\begin{proof}
				Let us take a monomial ideal $E$ such that $[E]\in \overline{\G_m^2 \cdot [I]}$.
				The claim is preserved under the field extension so we may assume that $\kk$ is algebraically closed, thus perfect.
				Now we can use the classical result by Kempf \cite[Corollary 4.3]{Kempf} to say that if the monomial ideal $[E]$ is inside the closure of the orbit $\G_m^2 \cdot [I]$, then there exists a one-parameter subgroup $\G_m \cong G \subset \G_m^2$ such that $E \in \overline{G \cdot [I]}$. We choose a cocharacter $\psi\colon \G_m \to \G_m^2$ such that $\psi(\G_m)=G$ with the action of $\G_m$:
				\begin{align}
						   \G_m \times \Hilb^d(\A^2) &\to \Hilb^d(\A^2),\\
						 (t,[I]) &\mapsto (t\cdot [I]),
				\end{align}
				such that \[ \lim_{t\to 0} t\cdot [I] = [E]. \]
				As a consequence, $[I]$ lies in Białynicki-Birula cell associated to $\psi$ and centred at $E$.
				There are two cases, either the cocharacter $\psi$ contracts $y$-axis or it expands it. In the first case we consider the family constructed in Theorem \ref{theorem-abb}
				$$\begin{tikzcd}
					\V(I_t(\mathcal M_E))_{|S} \arrow[r] \arrow[d, "\pi_{|S}"] & \V(I_t(\mathcal M_E)) \arrow[d,"\pi"] \\
					S \arrow[r]                                           & B_E,                           
				\end{tikzcd}$$
				where $S$ is the Białynicki-Birula cell $B_E^{+,\psi}$. For simplicity we will denote $\V(I_t(\mathcal M_E))$ as $V(I)$.
				In the second case we follow the discussion of Diagram ~\ref{thicken-x-axis} and Theorem \ref{theorem-abb-} and consider 
				$$\begin{tikzcd}
					\V(I_t(\mathcal M_E)+y^{tm_t})_{|S} \arrow[r] \arrow[d, "{\pi_{-}}_{|S}"] & \V(I_t(\mathcal M_E)+y^{tm_t}) \arrow[d, "\pi_-"] \\
					S \arrow[r]                                           & B_E,                            
				\end{tikzcd}$$
				again $S$ is the Białynicki-Birula cell $B_E^{+,\psi}$, and we denote $\V(I_t(\mathcal M_E)+y^{tm_t})$ as $V(I+y^{tm_t})$. 
				We will consider both cases simultaneously.
				\vskip 0.2 cm
				Let us take $\mathbf S \subset B_E$, a first-order thickening of $S$ in $B_E$ such that $T_b \mathbf S = T_b B_E$ for every point $b$ of $S$. Now we argue that $V(I)_{|\mathbf S}$ (resp. $V(I+y^{tm_t})$) induces a morphism to the Hilbert scheme.
				By Nakayama Lemma \cite[\href{https://stacks.math.columbia.edu/tag/051F}{Tag 051F}]{stacks-project}, the fact that the morphism $\pi_{|S}$ (resp. ${\pi_-}_{|S}$) is finite implies that the morphism $\pi_{|\mathbf S}$ (resp. ${\pi_-}_{|\mathbf S}$)  is also finite. By Schaps Theorem \ref{Hilbert-Burch} we get also flatness of $\pi_{|\mathbf S}$. In the second case we are considering $\pi_-$ instead of $\pi$.  But since $V(I+y^{tm_t})_{|S}$ is open and closed in $V(I)_{|S}$, its thickening is also open and closed, thus flatness of $\pi_{|\mathbf S}$ implies flatness of ${\pi_{-}}_{|\mathbf S}$.
				As a result, we get a morphism $\phi_{\mathbf S}\colon \mathbf S \to \Hilb^d(\A^2)$.
				\vskip 0.2 cm
				Let us denote the embedding of $S$ into $B_E$ by $i\colon S \to B_E$.
				Then the morphism $\phi_{\mathbf S}$ obtained above induces the following commutative diagram:
				$$\begin{tikzcd}
					i^* TB_E \arrow[d] \arrow[r] & T \Hilb^d(\A^2) \arrow[d] \\
					S \arrow[r, "\phi^+"]        & \Hilb^d(\A^2),            
				\end{tikzcd}$$
				where $\phi^+$ is the morphism described in Theorem \ref{theorem-abb} (resp. Theorem \ref{theorem-abb-}).
				Thus we get a morphism of sheaves:
				$$d\phi\colon  i^{*}TB_E \to (\phi^{+})^*T\Hilb^d(\A^2)$$
				The morphism $d\phi$  restricted to $0 \in B_E$ is the isomorphism of tangent spaces $d\phi_E$ described in Theorem \ref{iso-tangent}. As a consequence $d\phi$ is an isomorphism on some open subscheme $U$ of $S$.
				Note that since $\pi$ (resp $\pi_-$) is $\G_m^2$-equivariant, so is the morphism $d\phi$. And thus $U$ is a $\G_m^2$-stable subscheme of $S$. Consequently it contains the preimage of $[I]$ under the map $\phi^+$. 
				\vskip 0.2 cm
				This way we have obtained a map of tangent spaces at any point $b\in B_E$ such as in the statement of the theorem.
				The rest is to extend this morphism to a morphism of spectra of complete local rings. The argument follows identically as in Theorem \ref{infinitesimal-phi}.
			\end{proof}
	
	\section{Where the Dream scenario holds} \label{chapter8}
		So far we have shown that our strategy shows that our family of matrices induces maps between the Białynicki-Birula cells. In this section, we see examples when the dream scenario from Section \ref{chapter5} occurs and it is not a direct consequence of Theorem \ref{theorem-abb}.
		%For every $d$ even and larger than $4$ there is a map to $\Hilb^d(\A^2)$ that is coming from such an example 
		\begin{exm} \label{example-200}
			Let us consider the monomial ideal $E=(y^2,x^3)$. Let us plot the weights of a few of the parameters.
			$$	\includegraphics[scale=0.5]{example200.png} $$

		We see that there is no half-plane through the origin containing all the weights. This means that no open neighbourhood of the origin is fully contained in a Białynicki-Birula cell for some ~$\psi$. Thus the potential existence of the rational map to the Hilbert scheme cannot be predicted by Theorem \ref{theorem-abb}. 		
	
		Let us consider the spread-out matrix from Definition \ref{matrix-general}
		$$\mathcal M = \begin{pmatrix}
			y^{2}+a_{1,1}&a_{1,2}&a_{1,3}\\
			-x+a_{2,1}&1&0\\
			a_{3,1}&-x&1\\
			a_{4,1}&0&-x\\
		\end{pmatrix}.$$
		Since we are interested only in the degree of the resultant of the first and last minor, we can consider only the variables $b_{i,j}:=a_{i,j,1}$, since for every term in that the resultant that does contain $a_{i,j,0}$ there is always one with it replaced by $a_{i,j,1}$ and of degree one higher with respect to $y$.
		Thus, we can assume that our matrix is
		$$\mathcal M = \begin{pmatrix}
			y^{2}+y\,b_{1,1}&y\,b_{1,2}&y\,b_{1,3}\\
			y\,b_{2,1}-x&1&0\\
			y\,b_{3,1}&-x&1\\
			y\,b_{4,1}&0&-x\\
			\end{pmatrix}.$$

		The last minor is
		$$f_3= x^{2}y\,b_{1,3} -x\,y^{2}b_{2,1}b_{1,3}+x\,y\,b_{1,2}-y^{2}b_{2,1}b_{1,2}+-y^{2}b_{3,1}b_{1,3}+y^{2}+y\,b_{1,1}$$
		and the first one
		$$ f_0 =-x^{3} + x^{2}y\,b_{2,1}+x\,y\,b_{3,1}+y\,b_{4,1}.
		$$
		Their Sylvester matrix with respect to the variable ~$x$ can be treated as a matrix of polynomials in $y$. Let us look at the matrix of the degrees of those polynomials.
		$$\begin{pmatrix}
			0&1&1&1&-1\\
			-1&0&1&1&1\\
			1&2&2&-1&-1\\
			-1&1&2&2&-1\\
			-1&-1&1&2&2\\
		\end{pmatrix},$$
		where we denote degree of $0$ as $-1$.
		From this, by considering the Leibniz formula for determinants, we see that the degree of the resultant with respect to $y$ is lower or equal to $7$.
		Let us consider the $2\times 2$ minors as highlighted below
		$$\begin{pNiceMatrix}

				0 \Block[fill=red!60,rounded-corners]{1-2}{}&1&1&1&-1\\
			-1&0&1&1&1\\
			1 \Block[fill=red!60]{1-2}{} &2&2&-1&-1\\
			-1&1&2&2&-1\\
			-1&-1&1&2&2\\
		\end{pNiceMatrix}.$$
		The minor is
		 $$\begin{pmatrix}
			\vphantom{\left\{-5\right\}}-1&y\,b_{2,1}\\
			\vphantom{\left\{-7\right\}}y\,b_{1,3}&-y^{2}b_{2,1}b_{1,3}+y\,b_{1,2}\\
		\end{pmatrix}.$$
		The determinant of that minor has the degree equal to $1$ with respect to $y$.
		This means that the contribution of the first two columns to the degree of the resultant is at most $1$. The only way to get the degree $6$ out of the last three columns is by multiplying the following entries:
			$$\begin{pNiceMatrix}
					0 \Block[fill=red!30,rounded-corners]{1-1}{}&1&1&1&-1\\
				-1&0 \Block[fill=red!30]{1-1}{}&1&1&1\\
				1&2&2 \Block[fill=blue!30]{1-1}{}&-1&-1\\
				-1&1&2&2 \Block[fill=blue!30]{1-1}{}&-1\\
				-1&-1&1&2&2 \Block[fill=blue!30,rounded-corners]{1-1}{}\\
			
		\end{pNiceMatrix}.$$
		but this way we obtain the full degree of resultant to be $6$.
		We have shown that the degree of resultant is not greater than $6$, and thus by Lemma \ref{resultant-flat} and Lemma \ref{lemma-finitness} we obtain a rational map to the Hilbert scheme.
				\end{exm}		
		An almost identical argument can be made to show that we get a rational map for the following monomial ideals:	
		$$\left(y^{5},\,x\,y^{4},\,x^{2}y^{3},\,x^{3}\right)$$	
		$$\left(y^{8},\,x\,y^{6},\,x^{2}y^{4},\,x^{3}\right)$$
		$$\left(y^{11},\,x\,y^{8},\,x^{2}y^{5},\,x^{3}\right)$$
		$$\left(y^{14},\,x\,y^{10},\,x^{2}y^{6},\,x^{3}\right)$$

		\begin{quest} \label{quest}
			What are the monomial ideals $E$ such that the construction from Definition \ref{matrix-general} induces a rational map to the Hilbert scheme, but there is no open Białynicki-Birula cell containing $E$. 
			Can one distinguish specific infinite classes of such ideals?
		\end{quest}
		%For example strategy of finding an example and enlarging it by adding $0$ at the end of the sequence doesn't work in general.
		%\begin{exm} \label{example-weird}
		%	$(3,1,1)$ is an example and $(3,1,1,0)$ is not good [Problem with that is M2 can easily say that, but otherwise there are a lot of computations]
		%\end{exm}
		
	\bibliographystyle{alpha}
	\bibliography{references}
	\bibliographystyle{alpha}
\end{document}